\documentclass[11pt]{amsart}
\usepackage{amsfonts}
\usepackage{amsthm}
\usepackage{geometry}
\usepackage{amsfonts}
\usepackage{amsmath}
\usepackage{dsfont}
\usepackage{graphics}

\newcommand{\EEQ}{\end{equation}}
\newcommand{\rfb}[1]{\mbox{\rm
   (\ref{#1})}\ifx\undefined\stillediting\else:\fbox{$#1$}\fi}

                         \newcommand{\ud}     {{\rm d}}
\newcommand{\bt}{\begin{Theorem}}
\newcommand{\et}{\end{Theorem}}
\newcommand{\br}{\begin{remark}}
\newcommand{\er}{\end{remark}}
\newcommand{\bc}{\begin{Corollary}}
\newcommand{\ec}{\end{Corollary}}
\newcommand{\el}{\end{Lemma}}
\newcommand{\bd}{\begin{definition}}
\newcommand{\ed}{\end{definition}}

\newcommand{\N}  {\mathbb{N}}
\newcommand{\R}  {\mathbb{R}}

\newcommand{\mm}    {{\hbox{\hskip 0.5pt}}}

\newcommand{\bluff} {{\hbox{\raise 15pt \hbox{\mm}}}}

\newfont{\Blackboard}{msbm10 scaled 1200}

\newfont{\roma}{cmr10 scaled 1200}

\def\CC{\rm \hbox{C\kern-.56em\raise.4ex
         \hbox{$\scriptscriptstyle |$}\kern+0.5 em }}
\newtheorem{corollary}{Corollary}[section]
\newtheorem{definition}[corollary]{Definition}

\newtheorem{remark}[corollary]{Remark}
\numberwithin{equation}{section}
\newtheorem{lem}{Lemma}[section]
\newtheorem{prop}{Proposition}[section]

\newtheorem{thm}{Theorem}[section]

\def\ds{\displaystyle}
\newcommand{\re}{\mathrm{Re}}

\begin{document}
\thispagestyle{empty}
  
\title[Stabilization for beam equation with a degenerated Kelvin-Voigt damping]{Stabilization for Euler-Bernoulli beam equation with a local degenerated Kelvin-Voigt damping}

\author{Fathi Hassine}
\address{UR Analysis and Control of PDEs, UR 13ES64, Department of Mathematics, Faculty of Sciences of Monastir, University of Monastir, Tunisia}
\email{fathi.hassine@fsm.rnu.tn}

\begin{abstract}
We consider the Euler-Bernoulli beam equation with a local Kelvin-Voigt dissipation type in the interval $(-1,1)$. The coefficient damping is only effective in $(0,1)$ and is degenerating near the $0$ point with a speed at least equal to $x^{\alpha}$ where $\alpha\in(0,5)$. We prove that the semigroup corresponding to the system is polynomially stable and the decay rate depends on the degeneracy speed $\alpha$.
\end{abstract}    


\maketitle

\tableofcontents
 
\section{Introduction}
In the last two decades, owing to the large applications of smart materials in modern technology, there has been an increasing research on elastic systems with viscoelastic damping \cite{BIW,BSBSM}. When the smart materials are added into the elastic structures, the Young’s modulus, the mass density and the damping coefficients are changed accordingly. This passive method, on the one hand, makes the distributed control practically applicable but on the other hand, brings some new mathematical challenges which attract an increasing research interests. 

The wave equation involving Kelvin-Voigt damping was extensively studied in last decades. For the high dimensional case when the damping is localized in a suitable open subset, of the domain under consideration, it is shown that the energy of the system decays polynomially or exponentially depending on the regularity of the coefficient damping \cite{liu-rao2,tibou2,burq2} and it decays logarithmically \cite{AHR2,burq2} with an arbitrary localized. In one dimensional case the associated semigroup is analytic when the entire medium is of the Kelvin-Voigt type \cite{huang2} however a luck of exponential stability was shown when the damping is localized \cite{chen-liu-liu,hassine1,hassine-souayeh}. Moreover, only a polynomial decay rate of the energy is shown with a sharp decay rate equal to $t^{-2}$ \cite{hassine1,liu-rao1}. In the other hand a uniform decay rate still holds if the damping coefficient is smooth enough \cite{liu-liu2,zhang}. 

The Euler-Bernoulli plate equation with a Kelvin-Voigt damping were considered by several authors. The associated system operator generates an analytic semigroup if the damping is global \cite{chen-liu-liu,liu-liu2} and an interesting Riesz basis property was developed for such a system in \cite{JTW}. While in multi-dimensional case this property is not anymore valid and we obtain by means of Carleman estimate a logarithmic decay rate only \cite{hassine3} (see also the transmission wave/plate equation \cite{hassine2}) in the one-dimensional case the semigroup still enjoys the uniform decay rate of the energy \cite{hassine1,liu-liu} (see also the transmission wave/plate equation \cite{hassine4}) if the damping is localized. 

For the spectral analysis it was shown in	\cite{ZG} for the 1-d case that the essential spectrum of the system operator is identified to be an interval on the left real axis. Moreover, under some assumptions on the coefficients, it was shown that the essential spectrum also contains continuous spectrum only, and the point spectrum consists of isolated eigenvalues of finite algebraic multiplicity. The asymptotic behavior of eigenvalues is presented when the coefficients are regular (say $\mathcal{C}^{3}$ on the close domain) and bounded by bellow by some a non negative constant.

The remainder of the paper is organized as follows. In section \ref{IBKV} we stat the problem and the main result. In section \ref{PLBKV} we give some lemmas which are useful for the proofs of Theorem \ref{IBKV3}. Sections \ref{PSBKV} and \ref{DTBKV} are devoted to prove Theorem \ref{IBKV3}.
\section{Problem statement and main result}\label{IBKV}
The system that we are concerned with is the following Euler–Bernoulli beam equation clamped at two boundaries with
internal Kelvin–Voigt damping
\medskip
\begin{equation}\label{IBKV1}
\left\{
\begin{array}{ll}
\ds\ddot{u}(x,t)+u''''(x,t)+\left(b(x)\dot{u}''(x,t)\right)''=0&(x,t)\in(-1,1)\times(0,+\infty),
\\
u(-1,t)=u(1,t)=0&t\in(0,+\infty),
\\
u'(-1,t)=u'(1,t)=0&t\in(0,+\infty),
\\
\ds u(x,0)=u^{0}(x),\;\dot{u}(x,0)=u^{1}(x)&x\in(-1,1),
\end{array}
\right.
\end{equation}
where $b$ is the coefficient damping which assumed to be a non-negative function such that
$$
b(x)=\left\{\begin{array}{ll}
0&\text{if } x\in(-1,0)
\\
a(x)&\text{if } x\in(0,1).
\end{array}\right.
$$
The natural energy of system \eqref{IBKV1} is given by
$$
E(t)=\frac{1}{2}\int_{-1}^{1}\left(|\dot{u}(x,t)|^{2}+|u''(x,t)|^{2} \right)\,\ud x
$$
and it is dissipated according to the following law
\begin{equation*}
\frac{\ud}{\ud t}E(t)=-\int_{-1}^{1}b(x)\,|\dot{u}''(x,t)|^{2}\,\ud x,\;\forall\,t>0.
\end{equation*}
Let $\ds\mathcal{H}=H_{0}^{2}(-1,1)\times L^{2}(-1,1)$ be the Hilbert space endowed with the inner product define for $(u,v),\,(\tilde{u},\tilde{v})\in\mathcal{H}$ by
$$ 
\left\langle (u,v),(\tilde{u},\tilde{v})\right\rangle_{\mathcal{H}}=\int_{-1}^{1}u''(x)\,.\,\overline{\tilde{u}}''(x)\,\ud x+\int_{-1}^{1}v(x)\,.\,\overline{\tilde{v}}(x)\,\ud x.
$$
By setting $U(t)=(u(t),v(t))$ and $U^{0}=(u^{0},u^{1})$ we can rewrite system \eqref{IBKV1} as a first order differential equation as follows
\begin{equation}\label{IBKV2}
\dot{U}(t)=\mathcal{A}U(t),\qquad U(0)=U^{0}\in\mathcal{D}(\mathcal{A}),
\end{equation}
where
$$
\mathcal{A}(u,v)=(v,-u''''-(bv'')''),
$$
with
$$
\mathcal{D}(\mathcal{A})=\left\{(u,v)\in\mathcal{H}:\;(v,u''''+(bv'')'')\in \mathcal{H}\right\}. 
$$
Following to \cite{chen-liu-liu,hassine1,hassine3,liu-liu} the operator $\mathcal{A}$ is a generator of a $C_{0}$-semigroup of contraction $e^{t\mathcal{A}}$ on the Hilbert space $\mathcal{H}$. Therefore for every $U^{0}\in\mathcal{D}(\mathcal{A})$ the Cauchy problem \eqref{IBKV2} admits a unique solution $U(t)\in\mathcal{C}([0,+\infty),\mathcal{D}(\mathcal{A}))\cap\mathcal{C}^{1}([0,+\infty),\mathcal{H})$. Moreover, it is also shown that the imaginary axis is a subset of the resolvent set of $\mathcal{A}$ then the semigroup generated by $\mathcal{A}$ is strongly stable, i.e for any $U^{0}\in\mathcal{H}$ we have
$$
\lim_{t\rightarrow+\infty}\|e^{t\mathcal{A}}U^{0}\|_{\mathcal{H}}=0.
$$

The difficulty of the problem we consider consists in the fact that we don't assume that the localizing coefficient $b$ satisfies a condition of the form $a(x)>a_{0}$ in $(0,1)$, for some $a_{0}>0$, as in most of the previous literature \cite{hassine1, hassine4,liu-liu,liu-rao2} where an exponential stability was proven. Besides, a more general assumption on the coefficient was taken into account for the muti-dimensional case in \cite{hassine2,hassine3} where a logarithmic decay rate of the energy via a Carleman estimate was proven.

As far as we know the only paper which tackles the plate equation with a local degenerated dissipation is the one of Guzm\`an and Tucsnak in \cite{GT}. In fact the authors consider the following initial and boundary value problem modeling the damped vibrations of a simply supported plate
$$
\left\{\begin{array}{ll}
u''+\Delta^{2} u+a(x)u'=0&\text{in }\Omega\times(0,+\infty)
\\
u=\Delta u=0&\text{on }\Gamma\times(0,+\infty)
\\
u(0)=u^{0},\;u'(0)=u^{1}&\text{in }\Omega
\end{array}\right.
$$
where $\Omega$ is an open bounded subset of $\R^{2}$ with boundary $\Gamma$ and $a\in L^{\infty}(\Omega)$ be such that $a(x)\geq 0$ for almost all $x\in\Omega$ and where it is assumed
$$
\int_{\omega}\frac{\ud x}{a^{p}(x)}<\infty
$$ 
for some $p>0$ with $\omega$ an arbitrary non-empty subset of $\Omega$. They show that the usual observability inequality for the undamped problem implies polynomial decay estimates for the damped problem with a rate equal to $t^{-2p}$.

The aim of this paper is to describe the decay rate of the semigroup associate to \eqref{IBKV1} when the damping term $a$ degenerates near $0$ as $x^{\alpha}$, precisely we have the following 
\begin{thm}\label{IBKV3}
Assuming that $a$ is a non-negative and an increasing function in such a way that $a\in\mathcal{C}([0,1))\cap\mathcal{C}^{2}((0,1))$ and there exist $\alpha\in(0,5)$ and $\kappa\geq 0$ such that
\begin{equation}\label{IBKV4}
	\lim_{x\to 0^{+}}\frac{a(x)}{x^{\alpha}}=\kappa.\tag{A \addtocounter{equation}{1}\theequation}
\end{equation}
Then, the semigroup $e^{t\mathcal{A}}$ associated to system \eqref{IBKV1} is polynomially stable precisely for every $\varepsilon>0$ there exists $C>0$ such that
$$
\|e^{t\mathcal{A}}(u^{0},u^{1})\|_{\mathcal{H}}^{2}\leq\frac{C}{(t+1)^{(\tau(\alpha)+\varepsilon)k}}\left\|(u^{0},u^{1})\right\|_{\mathcal{D}(\mathcal{A}^{k})}^{2}\quad\forall\,(u^{0},u^{1})\in\mathcal{D}(\mathcal{A}^{k})\;\forall\,t\geq 0.
$$
where
$$
\tau(\alpha)=\left\{\begin{array}{ll}
\ds\frac{5-\alpha}{3-\alpha}&\ds\text{if }\alpha\in\left(0,\frac{5}{3}\right]
\\
\ds\frac{5+\alpha}{1+\alpha}&\ds\text{if }\alpha\in\left(\frac{5}{3},3\right]
\\
\ds\frac{4}{\alpha-1}&\text{if }\alpha\in(3,5).
\end{array}\right.
$$
\end{thm}
In the case of the wave equation involving a Kelvin-Voigt damping with a coefficient that satisfying the same assumption as \eqref{IBKV4} with $\alpha\in(0,1)$ we proved in \cite{hassine7} that the semigroup associated to the system is polynomially stable with a rate equal to $\ds\frac{3-\alpha}{2(1-\alpha)}$. Thus as $\alpha$ goes to $1^{-}$ the order of polynomial stability goes to $+\infty$ which is consistent with the exponential stability \cite{LZ} when $\alpha=1$ and as $\alpha$ goes to $0^{+}$ the order of polynomial stability $\ds\frac{3-\alpha}{2(1-\alpha)}$ goes to $\ds\frac{3}{2}$ which is different to the optimal order of stability \cite{hassine1,liu-rao1} which is equal to $2$ when $\alpha=0$.

\begin{figure}[htbp]
	\centering
		\includegraphics{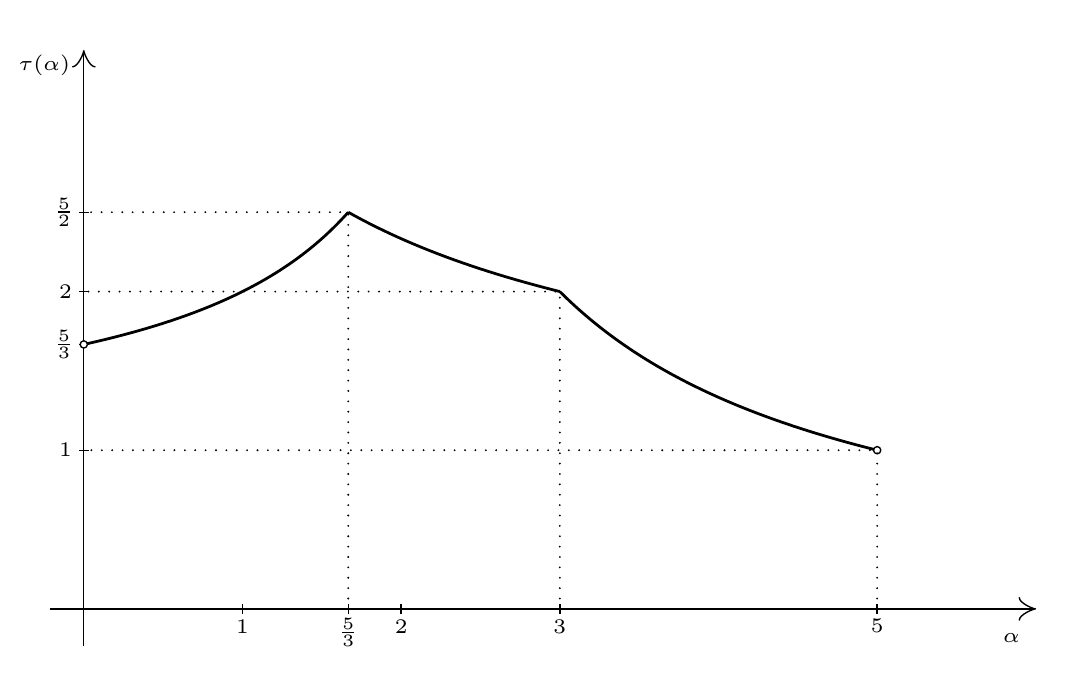}
	\caption{The dependence of the decay rate as a function of $\alpha$.}
	\label{fig1:PSBKV}
\end{figure}

Unlike the wave equation whose decay rate is increasing as the regularity of the coefficient damping becomes more and more regular, the plate equation, as illustrated in Figure \ref{fig1:PSBKV}, is with a decay rate that increases in $\ds\Big(0,\frac{5}{3}\Big]$ and decreases in $\ds\Big[\frac{5}{3},5\Big)$ and this is due essentially to the higher order boundary terms at the interface. However, the uniform stability is shown when the discontinuity of material properties take place and this is corresponding to $\alpha=0$. We believe that the purely elastic waves are almost completely reflected at the interface due to the ``strong'' degeneracy of the coefficient damping and thus are very weakly dissipated by the damping mechanism. Noting that the case $\alpha\geq 5$ still an interesting open question since an answer to this question would give a whole description of the behavior of the decay rate as $\alpha$ goes to $+\infty$.

Our approach consists to use a local analysis approach \cite{hassine7,LZ} combined with the classical iterative method \cite{hassine6} in order to deal with the higher order of the operator. This method can be used method can be applied for other PDE's with any higher order operator.
\section{Some preliminary Lemmas}\label{PLBKV}
The proof of Theorem \ref{IBKV3} rely on the Hardy inequalities and an interpolation inequality presented below. Let's first recall the following Hardy's inequality.
\begin{lem}\cite{stepanov}\label{PSBKV26}
Let $L>0$ and $\rho_{1},\,\rho_{2}>0$ be two weight functions defied on $(0,L)$. Then the following conditions are equivalent:
\begin{equation}\label{PSBKV28}
\int_{0}^{L}\rho_{1}(x)|Tf(x)|^{2}\,\ud x\leq C\int_{0}^{L}\rho_{2}(x)|f(x)|^{2}\,\ud x,
\end{equation}
and
$$
K=\sup_{x\in(0,L)}\left(\int_{0}^{L-x}\rho_{1}(x)\,\ud x\right)\left(\int_{L-x}^{L}\left[\rho_{2}(x)\right]^{-1}\,\ud x\right)<\infty
$$
where $\ds Tf(x)=\int_{0}^{x}f(s)\,\ud x$. Moreover, the best constant $C$ in \eqref{PSBKV28} satisfies $K\leq C\leq 2K$.
\end{lem}
By taking $\rho_{1}(x)=x^{\beta}$ and $\rho_{2}(x)=x^{\alpha}$ in Lemma \ref{PSBKV26}, we have the following
\begin{lem}\cite[Lemma 2.3]{LZ}\label{PSBKV27}
Let $x^{\frac{\alpha}{2}}y'\in L^{2}(0,1)$ satisfy $y(1)=0$. Assume that $\beta\geq \alpha-2$ when $\alpha>1$ and $\beta>-1$ when $0\leq\alpha<1$. Then there exists positive constant $C$ independent of $y$, such that
$$
\int_{0}^{1}x^{\beta}|y(x)|^{2}\,\ud x \leq C\int_{0}^{1}x^{\alpha}|y'(x)|^{2}\,\ud x.
$$
\end{lem}
To deal with the case $\alpha=1$ we need the following Hardy's inequality.
\begin{lem}\cite[Lemma 3.3]{GLL}\label{PSBKV106}
Let $x^{\frac{\alpha}{2}}y'\in L^{2}(0,1)$ satisfy $y(1)=0$. Assume that $0\leq \alpha\leq 2$. Then there exists positive constant $C$ independent of $y$, such that
$$
\int_{0}^{1}|y(x)|^{2}\,\ud x \leq C\int_{0}^{1}x^{\alpha}|y'(x)|^{2}\,\ud x.
$$
\end{lem}
The following lemma is an interpolation inequality that takes into account the size of the domain. 
\begin{lem}\label{PSBKV112}
There exists $K\geq 0$ such that for every $c,\,d\in\R$ with $c<d$ we have
$$
\|f'\|_{L^{2}(c,d)}\leq K(d-c)\|f''\|_{L^{2}(c,d)}+(d-c)^{-1}\|f\|_{L^{2}(c,d)}\quad \forall\,f\in H^{2}(c,d).
$$
\end{lem}
\begin{proof}
It is sufficient to prove this result for real-valued function. We assume, that $c=0$ and $d=1$. If $\ds 0<\xi<\frac{1}{3}$ and $\ds\frac{2}{3}<\eta<1$, then there exists $\lambda\in(\xi,\eta)$ such that
$$
|f'(\lambda)|=\left|\frac{f(\eta)-f(\xi)}{\eta-\xi}\right|\leq 3(|f(\xi)|+|f(\eta)|).
$$
It follows that for any $x\in[0,1]$ we have 
\begin{align*}
|f'(x)|&=\left|f'(\lambda)+\int_{\lambda}^{x}f''(t)\,\ud t\right|
\\
&\leq 3|f(\xi)|+3|f(\eta)|+\int_{0}^{1}|f''(t)|\,\ud t.
\end{align*}
Integrating the above inequality with respect to $\xi$ over $(0,\frac{1}{3})$ and with respect to $\eta$ over $(\frac{2}{3},1)$, we obtain
\begin{align*}
\frac{1}{9}|f'(x)|&\leq\frac{1}{3}\int_{0}^{\frac{1}{3}}|f(\xi)|\,\ud\xi+\frac{1}{3}\int_{\frac{2}{3}}^{1}|f(\eta)|\,\ud\eta+\frac{1}{9}\int_{0}^{1}|f''(t)|\,\ud t
\\
&\leq\int_{0}^{1}|f(t)|\,\ud t+\frac{1}{9}\int_{0}^{1}|f''(t)|\,\ud t.
\end{align*}
Hence by H\"older's inequality
$$
\int_{0}^{1}|f'(x)|^{2}\,\ud x\leq K\left(\int_{0}^{1}|f(t)|^{2}\,\ud t+\int_{0}^{1}|f''(t)|^{2}\,\ud t\right).
$$
It follows by a change  of variable that for any finite interval $(c,d)$
$$
\int_{c}^{d}|f'(t)|^{2}\,\ud t\leq K\left((d-c)^{2}\int_{c}^{d}|f''(t)|^{2}\,\ud t+(d-c)^{-2}\int_{c}^{d}|f(t)|^{2}\,\ud t\right).
$$
This complete the proof.
\end{proof}
\section{Polynomial stability}\label{PSBKV}
This section is devoted to prove polynomial stability of the solutions of \eqref{IBKV1} as given in Theorem \ref{IBKV3}. The idea is to estimate the energy norm and boundary terms at the interface by the local viscoelastic damping. The difficulty is to deal with the higher order boundary terms at the interface so that the energy on $(-1,0)$ can be controlled by the viscoelastic damping on $(0,1)$. For this aim we recall the following result
\begin{prop}\cite[Theorem 2.4]{borichevtomilov}\label{PSBKV2}
Let $e^{tB}$ be a bounded $C_{0}$-semigroup on a Hilbert space $X$ with generator $B$ such that $i\R\in\rho(B)$. Then $e^{tB}$ is polynomially stable with order $\ds\frac{1}{\gamma}$ i.e. there exists $C>0$ such that
$$
\|e^{tB}u\|_{X}\leq \frac{C}{t^{\frac{1}{\gamma}}}\|u\|_{\mathcal{D}(B)}\quad\forall\,u\in\mathcal{D}(B)\;\forall\, t\geq 0,
$$
if and only if
$$
\limsup_{|\lambda|\rightarrow\infty}\|\lambda^{-\gamma}(i\lambda I-B)^{-1}\|_{X}<+\infty.
$$
\end{prop}
According to Proposition \ref{PSBKV2} we shall verify that for $\gamma=\tau(\tau)+\varepsilon$ there exists $C_{0}>0$ such that
\begin{equation}\label{PSBKV3}
\inf_{\substack{\|(u,v)\|_{\mathcal{H}}=1
\\
\lambda\in\R}}|\lambda|^{\gamma}\left\|i\lambda(u,v)-\mathcal{A}(u,v)\right\|_{\mathcal{H}}\geq C_{0}.
\end{equation}
Suppose that \eqref{PSBKV3} fails then there exist a sequence of real numbers $\lambda_{n}$ and a sequence of functions $(u_{n},v_{n})_{n\in\N}\subset\mathcal{D}(\mathcal{A})$ such that
\begin{eqnarray}
&\lambda_{n}\,\underset{n\,\rightarrow\,+\infty}{\longrightarrow}\,+\infty&
\\
&\big\|(u_{n},v_{n})\big\|_{\mathcal{H}}=1\quad \forall\,n\in\N,&\label{PSBKV4}
\\
&\lambda_{n}^{\gamma}\left\|i\lambda_{n}(u_{n},v_{n})-\mathcal{A}(u_{n},v_{n})\right\|_{\mathcal{H}}=o(1).&\label{PSBKV9}
\end{eqnarray}
Since, we have
\begin{equation}\label{PSBKV10}
\lambda_{n}^{\gamma}\re\langle(i\lambda_{n}-\mathcal{A})(u_{n},v_{n}),(u_{n},v_{n})\rangle=\int_{0}^{1}a|v_{n}''|^{2}\,\ud x
\end{equation}
then using \eqref{PSBKV4} and \eqref{PSBKV9} we obtain
\begin{equation}\label{PSBKV11}
\|a^{\frac{1}{2}}v_{n}''\|_{L^{2}(0,1)}=o(\lambda_{n}^{-\frac{\gamma}{2}}).
\end{equation}
Define
$$
u_{1,n}=u_{n}\mathds{1}_{(-1,0)},\quad v_{1,n}=v_{n}\mathds{1}_{(-1,0)},\quad u_{2,n}=u_{n}\mathds{1}_{(0,1)},\quad v_{2,n}=v_{n}\mathds{1}_{(0,1)}.
$$
Following to \eqref{PSBKV9} we have
\begin{align}
\lambda_{n}^{\gamma}(i\lambda_{n} u_{1,n}-v_{1,n})&=f_{1,n}\,\longrightarrow\,0\quad\text{in } H^{2}(-1,0),\label{PSBKV5}
\\
\lambda_{n}^{\gamma}(i\lambda_{n} u_{2,n}-v_{2,n})&=f_{2,n}\,\longrightarrow\,0\quad\text{in } H^{2}(0,1),\label{PSBKV18}
\\
\lambda_{n}^{\gamma}(i\lambda_{n} v_{1,n}+u_{1,n}'''')&=g_{1,n}\,\longrightarrow\,0\quad\text{in } L^{2}(-1,0),\label{PSBKV6}
\\
\lambda_{n}^{\gamma}(i\lambda_{n} v_{2,n}+M_{n}'')&=g_{2,n}\,\longrightarrow\,0\quad\text{in } L^{2}(0,1),\label{PSBKV7}
\end{align}
with the transmission conditions
\begin{align}
u_{1,n}(0)=u_{2,n}(0),\quad v_{1,n}(0)=v_{2,n}(0),\label{PSBKV19}
\\
u_{1,n}'(0)=u_{2,n}'(0),\quad v_{1,n}'(0)=v_{2,n}'(0),\label{PSBKV20}
\\
u_{1,n}''(0)=M_{n}(0),\label{PSBKV21}
\\
u_{1,n}'''(0)=M_{n}'(0),\label{PSBKV22}
\end{align}
where we have denoted by
\begin{equation}\label{PSBKV8}
M_{n}=u_{2,n}''+av_{2,n}''=(1+i\lambda_{n}a)u_{2,n}''-\lambda_{n}^{-\gamma}af_{2,n}''.
\end{equation}
By \eqref{PSBKV11} and \eqref{PSBKV18} we find
\begin{equation}\label{PSBKV12}
\|a^{\frac{1}{2}}u_{2,n}''\|_{L^{2}(0,1)}=o(\lambda_{n}^{-\frac{\gamma}{2}-1}).
\end{equation}
One multiplies \eqref{PSBKV6} by $(x+1)\overline{u}_{1,n}'$ integrating by parts over the interval $(-1,0)$ and uses \eqref{PSBKV4} and \eqref{PSBKV5}, one gets
\begin{equation}\label{PSBKV13}
3\|u_{1,n}''\|_{L^{2}(-1,0)}^{2}+\|v_{1,n}\|_{L^{2}(-1,0)}^{2}-|v_{1,n}(0)|^{2}-|u_{1,n}''(0)|^{2}+2\re\left((u_{1,n}'''(0)-u_{1,n}''(0))\overline{u}_{1,n}'(0)\right)=o(\lambda_{n}^{-\gamma}).
\end{equation}
One multiplies \eqref{PSBKV7} by $\overline{v}_{2,n}$ and \eqref{PSBKV6} by $\overline{v}_{1,n}$, integrating and summing up then by the transmission conditions \eqref{PSBKV19}-\eqref{PSBKV22} we get
\begin{multline}\label{PSBKV14}
i\lambda_{n}^{\gamma+1}\left(\|v_{1,n}\|_{L^{2}(-1,0)}^{2}+\|v_{2,n}\|_{L^{2}(0,1)}^{2}\right)+\lambda_{n}^{\gamma}\left(\int_{-1}^{0}u_{1,n}''\overline{v}_{1,n}''\,\ud x+\int_{0}^{1}u_{2,n}''\overline{v}_{2,n}''\,\ud x\right)
\\
+\lambda_{n}^{\gamma}\|a^{\frac{1}{2}}v_{2,n}'\|_{L^{2}(0,1)}^{2}=o(1).
\end{multline}
Taking the inner product of \eqref{PSBKV5} with $u_{1,n}$ in $H^{2}(-1,0)$ and the inner product of \eqref{PSBKV18} with $u_{2,n}$ in $H^{2}(0,1)$ and summing up, we obtain
\begin{equation}\label{PSBKV15}
i\lambda_{n}^{\gamma+1}\left(\|u_{1,n}''\|_{L^{2}(-1,0)}^{2}+\|u_{2,n}''\|_{L^{2}(0,1)}^{2}\right)-\lambda_{n}^{\gamma}\left(\int_{-1}^{0}v_{1,n}''\overline{u}_{1,n}''\,\ud x +\int_{0}^{1}v_{2,n}''\overline{u}_{2,n}''\,\ud x\right)=o(1).
\end{equation}
Subtracting \eqref{PSBKV15} from \eqref{PSBKV14} and taking the imaginary part of the expression then by \eqref{PSBKV11} we derive
\begin{equation}\label{PSBKV17}
\|u_{1,n}''\|_{L^{2}(-1,0)}+\|u_{2,n}''\|_{L^{2}(0,1)}-\left(\|v_{1,n}\|_{L^{2}(-1,0)}+\|v_{2,n}\|_{L^{2}(0,1)}\right)=o(1).
\end{equation}

We set 
$$
z_{n}^{\pm}(x)=\frac{i\lambda_{n}}{\sqrt{1+i\lambda_{n}a(x)}}\int_{x}^{\xi_{n}}\!\!\!\!\int_{t}^{\xi_{n}}v_{2,n}(s)\,\ud s\,\ud t\pm v_{2,n}(x)\quad\forall\,x\in(0,\xi_{n}),
$$
where $\xi_{n}\in(0,1)$ that would be specified later on.
\\
We derive $z_{n}^{\pm}$ two times then we get for every $x\in(0,\xi_{n})$ that
$$
z_{n}^{\pm '}(x)=\left(\frac{i\lambda_{n}}{\sqrt{1+i\lambda_{n}a(x)}}\right)'\int_{x}^{\xi_{n}}\!\!\!\!\int_{t}^{\xi_{n}}v_{2,n}(s)\,\ud s\,\ud t-\frac{i\lambda_{n}}{\sqrt{1+i\lambda_{n}a(x)}}\int_{x}^{\xi_{n}}v_{2,n}(s)\,\ud s\pm v_{2,n}'(x)
$$
and 
\begin{align}\label{PSBKV16}
z_{n}^{\pm ''}(x)=\left(\frac{i\lambda_{n}}{\sqrt{1+i\lambda_{n}a(x)}}\right)''\int_{x}^{\xi_{n}}\!\!\!\!\int_{t}^{\xi_{n}}v_{2,n}(s)\,\ud s\,\ud t-2\left(\frac{i\lambda_{n}}{\sqrt{1+i\lambda_{n}a(x)}}\right)'\int_{x}^{\xi_{n}}v_{2,n}(s)\,\ud s\nonumber
\\
+\frac{i\lambda_{n}}{\sqrt{1+i\lambda_{n}a(x)}}v_{2,n}(x)\pm v_{2,n}''(x)\nonumber
\\
=\left(\frac{i\lambda_{n}}{\sqrt{1+i\lambda_{n}a(x)}}\right)''\int_{x}^{\xi_{n}}\!\!\!\!\int_{t}^{\xi_{n}}v_{2,n}(s)\,\ud s\,\ud t-2\left(\frac{i\lambda_{n}}{\sqrt{1+i\lambda_{n}a(x)}}\right)'\int_{x}^{\xi_{n}}v_{2,n}(s)\,\ud s
\\
\pm\frac{i\lambda_{n}}{\sqrt{1+i\lambda_{n}a(x)}}z_{n}^{\pm}(x)\mp\frac{\lambda_{n}^{2}}{1+i\lambda_{n}a(x)}\int_{x}^{\xi_{n}}\!\!\!\!\int_{t}^{\xi_{n}}v_{2,n}(s)\,\ud s\,\ud t\pm v_{2,n}''(x).\nonumber
\end{align}
From \eqref{PSBKV18} and \eqref{PSBKV8} we have
\begin{equation}\label{PSBKV23}
v_{2,n}''(x)=i\lambda_{n} u_{2,n}''(x)-\lambda_{n}^{-\gamma}f_{2,n}''(x)=\frac{i\lambda_{n}}{1+i\lambda_{n}a(x)}M_{n}(x)-\frac{\lambda_{n}^{-\gamma}}{1+i\lambda_{n}a(x)}f_{2,n}''(x)\quad\forall\,x\in(0,\xi_{n}).
\end{equation}
By integrating two times \eqref{PSBKV7} we obtain
\begin{align}\label{PSBKV24}
\frac{\lambda_{n}^{2}}{1+i\lambda_{n}a(x)}\int_{x}^{\xi_{n}}\!\!\!\!\int_{t}^{\xi_{n}}v_{2,n}(s)\,\ud s\,\ud t=\frac{i\lambda_{n}}{1+i\lambda_{n}a(x)}\left(M_{n}(x)-M_{n}(\xi_{n})\right)+\frac{i\lambda_{n}}{1+i\lambda_{n}a(x)}(\xi_{n}-x)M_{n}'(\xi_{n})\nonumber
\\
-\frac{i\lambda_{n}^{1-\gamma}}{1+i\lambda_{n}a(x)}\int_{x}^{\xi_{n}}\!\!\!\!\int_{t}^{\xi_{n}}g_{2,n}(s)\,\ud s\,\ud t\quad\forall\,x\in(0,\xi_{n}).
\end{align}
Inserting \eqref{PSBKV23} and \eqref{PSBKV24} into \eqref{PSBKV16} we get
\begin{align}\label{PSBKV25}
z_{n}^{\pm ''}(x)=\left(\frac{i\lambda_{n}}{\sqrt{1+i\lambda_{n}a(x)}}\right)''\int_{x}^{\xi_{n}}\!\!\!\!\int_{t}^{\xi_{n}}v_{2,n}(s)\,\ud s\,\ud t-2\left(\frac{i\lambda_{n}}{\sqrt{1+i\lambda_{n}a(x)}}\right)'\int_{x}^{\xi_{n}}v_{2,n}(s)\,\ud s\nonumber
\\
\pm\frac{i\lambda_{n}}{1+i\lambda_{n}a(x)}z_{n}^{\pm}(x)\pm\frac{i\lambda_{n}}{1+i\lambda_{n}a(x)}M_{n}(\xi_{n})\mp\frac{i\lambda_{n}}{1+i\lambda_{n}a(x)}(\xi_{n}-x)M_{n}'(\xi_{n})\nonumber
\\
\pm\frac{i\lambda_{n}^{1-\gamma}}{1+i\lambda_{n}a(x)}\int_{x}^{\xi_{n}}\!\!\!\!\int_{t}^{\xi_{n}}g_{2,n}(s)\,\ud s\,\ud t\mp\frac{\lambda_{n}^{-\gamma}}{1+i\lambda_{n}a(x)}f_{2,n}''(x)\quad\forall\,x\in(0,\xi_{n}).
\end{align}
At this stage we would like to solve the previous second order ordinary differential equations on $z_{n}^{\pm}$ for this aim we define a sequence of functions as follow
\begin{align}\label{PSBKV47}
z_{n,0}^{\pm}(x)=\pm v_{2,n}(\xi_{n})\pm v_{2,n}'(\xi_{n})(x-\xi_{n})+\int_{x}^{\xi_{n}}\!\!\!\!\int_{\tau}^{\xi_{n}}\left(\frac{i\lambda_{n}}{\sqrt{1+i\lambda_{n}a(\rho)}}\right)''\int_{\rho}^{\xi_{n}}\!\!\!\!\int_{t}^{\xi_{n}}v_{2,n}(s)\,\ud s\,\ud t\,\ud \rho\,\ud\,\tau\nonumber
\\
-2\int_{x}^{\xi_{n}}\!\!\!\!\int_{\tau}^{\xi_{n}}\left(\frac{i\lambda_{n}}{\sqrt{1+i\lambda_{n}a(\rho)}}\right)'\int_{\rho}^{\xi_{n}}v_{2,n}(s)\,\ud s\,\ud\rho\,\ud\tau\mp M_{n}(\xi_{n})\int_{x}^{\xi_{n}}\!\!\!\!\int_{t}^{\xi_{n}}\frac{i\lambda_{n}}{1+i\lambda_{n}a(s)}\,\ud s\,\ud t\nonumber
\\
\pm M_{n}'(\xi_{n})\int_{x}^{\xi_{n}}\!\!\!\!\int_{t}^{\xi_{n}}\frac{i\lambda_{n}}{1+i\lambda_{n}a(s)}(\xi_{n}-s)\,\ud s\,\ud t\pm\int_{x}^{\xi_{n}}\!\!\!\!\int_{\tau}^{\xi_{n}}\frac{\lambda_{n}^{1-\gamma}}{1+i\lambda_{n}a(s)}\int_{\rho}^{\xi_{n}}\!\!\!\!\int_{t}^{\xi_{n}}g_{2,n}(s)\,\ud s\,\ud t\,\ud \rho\,\ud \tau\nonumber
\\
\mp\int_{x}^{\xi_{n}}\!\!\!\!\int_{t}^{\xi_{n}}\frac{i\lambda_{n}^{-\gamma}}{1+i\lambda_{n}a(s)}f_{2,n}''(s)\,\ud s\, \ud t\quad\forall\,x\in(0,\xi_{n}),
\end{align}
and for every $m\geq 0$
\begin{equation}\label{PSBKV48}
z_{n,m+1}^{\pm}(x)=\pm\int_{x}^{\xi_{n}}\!\!\!\int_{t}^{\xi_{n}}\frac{i\lambda_{n}}{1+i\lambda_{n}a(s)}z_{n,m}^{\pm}(s)\,\ud s\,\ud t\quad\forall\,x\in(0,\xi_{n}).
\end{equation}

Let $\delta>0$ that will be specified later on. Since $v_{2,n}'(1)=0$ then following to Lemma \ref{PSBKV27}, Lemma \ref{PSBKV106} and \eqref{PSBKV12} we have
\begin{align}\label{PSBKV29}
\|v_{2,n}'\|_{L^{2}\left(\frac{\lambda_{n}^{-\delta}}{2},\lambda_{n}^{-\delta}\right)}&\leq\|x^{-\frac{\beta_{0}}{2}}x^{\frac{\beta_{0}}{2}}v_{2,n}'\|_{L^{2}\left(\frac{\lambda_{n}^{-\delta}}{2},\lambda_{n}^{-\delta}\right)}\nonumber
\\
&\leq C\lambda_{n}^{\frac{\delta\beta_{0}}{2}}\|x^{\frac{\beta_{0}}{2}}v_{2,n}'\|_{L^{2}(0,1)}\nonumber
\\
&\leq C\lambda_{n}^{\frac{\delta\beta_{0}}{2}}\|x^{\frac{\alpha}{2}}v_{2,n}''\|_{L^{2}(0,1)}\nonumber
\\
&=o\left(\lambda_{n}^{\frac{\delta\beta_{0}}{2}-\frac{\gamma}{2}}\right)
\end{align}
where $\beta_{0}$ is satisfying
\begin{equation}\label{PSBKV31}
\beta_{0}\geq \alpha-2 \text{ if }\alpha>1,\quad \beta_{0}=0\text{ if }\alpha=1\quad\text{ and }\quad\beta_{0}>-1 \text{ if } 0\leq \alpha<1.
\end{equation}
Since in addition we have $v_{2,n}(1)=0$ then by Lemma \ref{PSBKV27}, Lemma \ref{PSBKV106} and \eqref{PSBKV11} we obtain
\begin{align}\label{PSBKV32}
\|v_{2,n}\|_{L^{2}\left(\frac{\lambda_{n}^{-\delta}}{2},\lambda_{n}^{-\delta}\right)}&\leq\|x^{-\frac{\beta'}{2}}x^{\frac{\beta'}{2}}v_{2,n}\|_{L^{2}	\left(\frac{\lambda_{n}^{-\delta}}{2},\lambda_{n}^{-\delta}\right)}\nonumber
\\
&\leq\lambda_{n}^{\frac{\delta\beta'}{2}}\|x^{\frac{\beta'}{2}}v_{2,n}\|_{L^{2}(0,1)}\nonumber
\\
&\leq C\lambda_{n}^{\frac{\delta\beta'}{2}}\|x^{\frac{\beta''}{2}}v_{2,n}'\|_{L^{2}(0,1)}\nonumber
\\
&\leq C\lambda_{n}^{\frac{\delta\beta'}{2}}\|x^{\frac{\alpha}{2}}v_{2,n}''\|_{L^{2}(0,1)}\nonumber
\\
&=o\left(\lambda_{n}^{\frac{\delta\beta'}{2}-\frac{\gamma}{2}}\right)
\end{align}
where $\beta'$, $\beta''$ and $\alpha$ are satisfying the following assumptions
\begin{align}
\beta'\geq \beta''-2 \text{ if }\beta''>1, \quad\beta'=0\text{ if }\beta''=1\quad\text{ and }\quad\beta'>-1 \text{ if } 0\leq \beta''<1\label{PSBKV33}
\\
\beta''\geq \alpha-2 \text{ if }1<\alpha<5, \quad \beta''=0\text{ if }\alpha=1 \quad\text{ and }\quad\beta''>-1 \text{ if } 0\leq \alpha<1.\label{PSBKV30}
\end{align}
Then following to \eqref{PSBKV32} we have 
\begin{equation}\label{PSBKV34}
\min_{x\in\left[\frac{\lambda_{n}^{-\delta}}{2},\lambda_{n}^{-\delta}\right]}|v_{2,n}(x)|\leq \sqrt{2}\lambda_{n}^{\frac{\delta}{2}}\|v_{2,n}\|_{L^{2}\left(\frac{\lambda_{n}^{-\delta}}{2},\lambda_{n}^{-\delta}\right)}=o\left(\lambda_{n}^{\frac{\delta}{2}(1+\beta')-\frac{\gamma}{2}}\right)
\end{equation}
and following to \eqref{PSBKV29} we have 
\begin{equation}\label{PSBKV35}
\min_{x\in\left[\frac{\lambda_{n}^{-\delta}}{2},\lambda_{n}^{-\delta}\right]}|v_{2,n}'(x)|\leq \sqrt{2}\lambda_{n}^{\frac{\delta}{2}}\|v_{2,n}'\|_{L^{2}\left(\frac{\lambda_{n}^{-\delta}}{2},\lambda_{n}^{-\delta}\right)}=o\left(\lambda_{n}^{\frac{\delta}{2}(1+\beta_{0})-\frac{\gamma}{2}}\right).
\end{equation}
From \eqref{PSBKV11} and \eqref{PSBKV12} we have
\begin{align}\label{PSBKV38}
\|M_{n}\|_{L^{2}\left(\frac{\lambda_{n}^{-\delta}}{2},\lambda_{n}^{-\delta}\right)}&\leq\|u_{2,n}''\|_{L^{2}\left(\frac{\lambda_{n}^{-\delta}}{2},\lambda_{n}^{-\delta}\right)}+\|av_{2,n}''\|_{L^{2}\left(\frac{\lambda_{n}^{-\delta}}{2},\lambda_{n}^{-\delta}\right)}\nonumber
\\
&\leq\|a^{-\frac{1}{2}}a^{\frac{1}{2}}u_{2,n}''\|_{L^{2}\left(\frac{\lambda_{n}^{-\delta}}{2},\lambda_{n}^{-\delta}\right)}+\max_{x\in\left[\frac{\lambda_{n}^{-\delta}}{2},\lambda_{n}^{-\delta}\right]}\left\{a^{\frac{1}{2}}(x)\right\}\|a^{\frac{1}{2}}v_{2,n}''\|_{L^{2}\left(\frac{\lambda_{n}^{-\delta}}{2},\lambda_{n}^{-\delta}\right)}\nonumber
\\
&\leq C\left(\max_{x\in\left[\frac{\lambda_{n}^{-\delta}}{2},\lambda_{n}^{-\delta}\right]}\left\{a^{-\frac{1}{2}}(x)\right\}\|a^{\frac{1}{2}}u_{2,n}''\|_{L^{2}\left(\frac{\lambda_{n}^{-\delta}}{2},\lambda_{n}^{-\delta}\right)}+\lambda_{n}^{\frac{-\delta\alpha}{2}}\|a^{\frac{1}{2}}v_{2,n}''\|_{L^{2}\left(\frac{\lambda_{n}^{-\delta}}{2},\lambda_{n}^{-\delta}\right)}\right)\nonumber
\\
&=\left(\lambda_{n}^{\frac{\delta\alpha}{2}-\frac{\gamma}{2}-1}+\lambda_{n}^{-\left(\frac{\delta\alpha}{2}+\frac{\gamma}{2}\right)}\right)o(1).
\end{align}
Then it follows that
\begin{equation}\label{PSBKV36}
\min_{x\in\left[\frac{\lambda_{n}^{-\delta}}{2},\lambda_{n}^{-\delta}\right]}|M_{n}(x)|\leq\sqrt{2}\lambda_{n}^{\frac{\delta}{2}}\|M_{n}\|_{L^{2}\left(\frac{\lambda_{n}^{-\delta}}{2},\lambda_{n}^{-\delta}\right)}=\left(\lambda_{n}^{\frac{\delta}{2}(\alpha+1)-\frac{\gamma}{2}-1}+\lambda_{n}^{\frac{\delta}{2}(1-\alpha)-\frac{\gamma}{2}}\right)o(1).
\end{equation}
From \eqref{PSBKV7} and \eqref{PSBKV32}, one gets
\begin{equation}\label{PSBKV37}
\|M_{n}''\|_{L^{2}\left(\frac{\lambda_{n}^{-\delta}}{2},\lambda_{n}^{-\delta}\right)}\leq\lambda_{n}\|v_{2,n}\|_{L^{2}\left(\frac{\lambda_{n}^{-\delta}}{2},\lambda_{n}^{-\delta}\right)}+\lambda_{n}^{-\gamma}\|g_{2,n}\|_{L^{2}\left(\frac{\lambda_{n}^{-\delta}}{2},\lambda_{n}^{-\delta}\right)}=\left(\lambda_{n}^{\frac{\delta\beta'}{2}-\frac{\gamma}{2}+1}+\lambda_{n}^{-\gamma}\right)o(1).
\end{equation}
Following Lemma \ref{PSBKV112}, \eqref{PSBKV38} and \eqref{PSBKV37}, we find
\begin{align}\label{PSBKV39}
\min_{x\in\left[\frac{\lambda_{n}^{-\delta}}{2},\lambda_{n}^{-\delta}\right]}|M_{n}'(x)|&\leq C\lambda_{n}^{\frac{\delta}{2}}\|M_{n}'\|_{L^{2}\left(\frac{\lambda_{n}^{-\delta}}{2},\lambda_{n}^{-\delta}\right)}\nonumber
\\
&\leq C\lambda_{n}^{\frac{\delta}{2}}\left(\lambda_{n}^{\delta}\|M_{n}\|_{L^{2}\left(\frac{\lambda_{n}^{-\delta}}{2},\lambda_{n}^{-\delta}\right)}+\lambda_{n}^{-\delta}\|M_{n}''\|_{L^{2}\left(\frac{\lambda_{n}^{-\delta}}{2},\lambda_{n}^{-\delta}\right)}\right)\nonumber
\\
&=\left(\lambda_{n}^{\frac{\delta}{2}\left(\alpha+3\right)-\frac{\gamma}{2}-1}+\lambda_{n}^{\frac{\delta}{2}\left(3-\alpha\right)-\frac{\gamma}{2}}+\lambda_{n}^{\frac{\delta}{2}(\beta'-1)-\frac{\gamma}{2}+1}+\lambda_{n}^{-(\gamma+\frac{\delta}{2})}\right)o(1).
\end{align}
We set 
\begin{align*}
p_{0}=\max\bigg\{&\frac{\delta}{2}(1+\beta')-\frac{\gamma}{2}, \,\frac{\delta}{2}(1+\beta_{0})-\frac{\gamma}{2}, \,\frac{\delta}{2}(\alpha+1)-\frac{\gamma}{2}-1, \,\frac{\delta}{2}(1-\alpha)-\frac{\gamma}{2},
\\
&\frac{\delta}{2}\left(\alpha+3\right)-\frac{\gamma}{2}-1,\, \frac{\delta}{2}\left(3-\alpha\right)-\frac{\gamma}{2},\, \frac{\delta}{2}(\beta'-1)-\frac{\gamma}{2}+1,\, -\left(\gamma+\frac{\delta}{2}\right)\bigg\}
\end{align*}
where $\beta'>-1$ and $\beta_{0}>-1$ (are small enough satisfying \eqref{PSBKV31}, \eqref{PSBKV33} and \eqref{PSBKV30})  and $\delta>0$ are suitably chosen in order to guarantee that 
\begin{equation}\label{PSBKV86}
p_{0}\leq 0.
\end{equation}
Henceforth we suppose that \eqref{PSBKV86} holds true.

Following to \eqref{PSBKV34}, \eqref{PSBKV35}, \eqref{PSBKV36} and \eqref{PSBKV39} we are able to choose $\xi_{n}\in\left[\frac{\lambda_{n}^{-\delta}}{2},\lambda_{n}^{-\delta}\right]$ such that
\begin{equation}\label{PSBKV73}
|v_{2,n}(\xi_{n})|+|v_{2,n}'(\xi_{n})|+|M_{n}(\xi_{n})|+|M_{n}'(\xi_{n})|=o\left(\lambda_{n}^{p_{0}}\right).
\end{equation}
Due to \eqref{PSBKV73} one has
\begin{equation}\label{PSBKV45}
|v_{2,n}'(\xi_{n})|.|x-\xi_{n}|=o\left(\lambda_{n}^{p_{0}-\delta}\right).
\end{equation}
Noting that
\begin{equation}\label{PSBKV75}
\ds\max_{\ds x\in\left[0,\lambda_{n}^{-\delta}\right]}\left\{\frac{1}{(1+\lambda_{n}^{2}a^{2}(x))^{\frac{1}{4}}}\right\}=O(1)
\end{equation}
and when $1>\alpha\delta$ we have $\ds 0<\delta<\frac{1+2\delta}{2+\alpha}<\frac{1+\delta}{1+\alpha}<\frac{2+\delta}{1+2\alpha}<\frac{1}{\alpha}$ and following to the assumption \eqref{IBKV4} we find
\begin{equation}\label{PSBKV40}
\ds\max_{\ds x\in\left[\lambda_{n}^{-\frac{1+2\delta}{2+\alpha}},\lambda_{n}^{-\delta}\right]}\left\{\frac{1}{(1+\lambda_{n}^{2}a^{2}(x))^{\frac{1}{4}}}\right\}=O\left(\lambda_{n}^{-\frac{1-\alpha\delta}{2+\alpha}}\right)\qquad\text{if }1>\alpha\delta
\end{equation}
\begin{equation}\label{PSBKV76}
\ds\max_{\ds x\in\left[\lambda_{n}^{-\frac{1+\delta}{1+\alpha}},\lambda_{n}^{-\delta}\right]}\left\{\frac{1}{(1+\lambda_{n}^{2}a^{2}(x))^{\frac{1}{2}}}\right\}=O\left(\lambda_{n}^{-\frac{1-\alpha\delta}{1+\alpha}}\right)\qquad\text{if }1>\alpha\delta
\end{equation}
and
\begin{equation}\label{PSBKV79}
\ds\max_{\ds x\in\left[\lambda_{n}^{-\frac{2+\delta}{1+2\alpha}},\lambda_{n}^{-\delta}\right]}\left\{\frac{1}{1+\lambda_{n}^{2}a^{2}(x)}\right\}=O\left(\lambda_{n}^{-\frac{2(1-\alpha\delta)}{1+2\alpha}}\right)\qquad\text{if }1>\alpha\delta.
\end{equation}
We derive from \eqref{PSBKV75}, \eqref{PSBKV40}, \eqref{PSBKV76} and \eqref{PSBKV79} that
\begin{align}\label{PSBKV77}
\int_{0}^{\xi_{n}}\frac{\ud s}{(1+\lambda_{n}^{2}a^{2}(s))^{\frac{1}{4}}}&\leq\int_{0}^{\lambda_{n}^{-\frac{1+2\delta}{2+\alpha}}}\frac{\ud s}{(1+\lambda_{n}^{2}a^{2}(s))^{\frac{1}{4}}}+\int_{\lambda_{n}^{-\frac{1+2\delta}{2+\alpha}}}^{\lambda_{n}^{-\delta}}\frac{\ud s}{(1+\lambda_{n}^{2}a^{2}(s))^{\frac{1}{4}}}\nonumber
\\
&=\left(\lambda_{n}^{-\frac{1+2\delta}{2+\alpha}}+\lambda_{n}^{-\frac{1-\alpha\delta}{2+\alpha}-\delta}\right)O(1)\nonumber
\\
&=O\left(\lambda_{n}^{-\frac{1+2\delta}{2+\alpha}}\right)\qquad \text{if }1>\alpha\delta,
\end{align}
\begin{align}\label{PSBKV78}
\int_{0}^{\xi_{n}}\frac{s^{-\beta}}{(1+\lambda_{n}^{2}a^{2}(s))^{\frac{1}{2}}}\ud s&\leq\int_{0}^{\lambda_{n}^{-\frac{1+\delta}{1+\alpha}}}\frac{s^{-\beta}}{(1+\lambda_{n}^{2}a^{2}(s))^{\frac{1}{2}}}\ud s+\int_{\lambda_{n}^{-\frac{1+\delta}{1+\alpha}}}^{\lambda_{n}^{-\delta}}\frac{s^{-\beta}}{(1+\lambda_{n}^{2}a^{2}(s))^{\frac{1}{2}}}\ud s\nonumber
\\
&=\left(\int_{0}^{\lambda_{n}^{-\frac{1+\delta}{1+\alpha}}}s^{-\beta}\,\ud s+\lambda_{n}^{-\frac{1-\alpha\delta}{1+\alpha}}\int_{\lambda_{n}^{-\frac{1+\delta}{1+\alpha}}}^{\lambda_{n}^{-\delta}}s^{-\beta}\,\ud s\right)O(1)\nonumber
\\
&=\left(\lambda_{n}^{-\frac{1+\delta}{1+\alpha}(1-\beta)}+\lambda_{n}^{-\frac{1-\alpha\delta}{1+\alpha}-\delta(1-\beta)}\right)O(1)\nonumber
\\
&=O\left(\lambda_{n}^{-\frac{1-\alpha\delta}{1+\alpha}-\delta(1-\beta)}\right)\qquad \text{if }1>\alpha\delta.
\end{align}
with $-1<\beta<1$ and
\begin{align}\label{PSBKV80}
\int_{0}^{\xi_{n}}\frac{\ud s}{1+\lambda_{n}^{2}a^{2}(s)}&\leq\int_{0}^{\lambda_{n}^{-\frac{2+\delta}{1+2\alpha}}}\frac{\ud s}{1+\lambda_{n}^{2}a^{2}(s)}+\int_{\lambda_{n}^{-\frac{2+\delta}{1+2\alpha}}}^{\lambda_{n}^{-\delta}}\frac{\ud s}{1+\lambda_{n}^{2}a^{2}(s)}\nonumber
\\
&=\left(\lambda_{n}^{-\frac{2+\delta}{1+2\alpha}}+\lambda_{n}^{-\frac{2(1-\alpha\delta)}{1+2\alpha}-\delta}\right)O(1)\nonumber
\\
&=O\left(\lambda_{n}^{-\frac{2+\delta}{1+2\alpha}}\right)\qquad \text{if }1>\alpha\delta.
\end{align}
By the Cauchy-Schwarz inequality, \eqref{PSBKV7}, \eqref{PSBKV75} and \eqref{PSBKV78} we obtain
\begin{align}\label{PSBKV41}
&\left|\int_{x}^{\xi_{n}}\!\!\!\!\int_{\tau}^{\xi_{n}}\frac{\lambda_{n}^{1-\gamma}}{1+i\lambda_{n}a(\rho)}\int_{\rho}^{\xi_{n}}\!\!\!\!\int_{t}^{\xi_{n}}g_{2,n}(s)\,\ud s\,\ud t\,\ud \rho\,\ud \tau\right|\nonumber
\\
&\leq\|g_{2,n}\|_{L^{2}(0,1)}\int_{x}^{\xi_{n}}\!\!\!\!\int_{\tau}^{\xi_{n}}\frac{\lambda_{n}^{1-\gamma}}{(1+\lambda_{n}^{2}a^{2}(\rho))^{\frac{1}{2}}}\int_{\rho}^{\xi_{n}}(\xi_{n}-t)^{\frac{1}{2}}\,\ud t\,\ud \rho\,\ud \tau\nonumber
\\
&=\left\{\begin{array}{ll}
o\left(\lambda_{n}^{1-\frac{7}{2}\delta-\gamma}\right)&\text{if }1\leq \alpha\delta
\\
o\left(\lambda_{n}^{1-\frac{1+\delta}{1+\alpha}-\frac{5}{2}\delta-\gamma}\right)&\text{if }1>\alpha\delta.
\end{array}\right.
\end{align}
Similarly, by Cauchy-Schwarz inequality, \eqref{PSBKV5}, \eqref{PSBKV75} and \eqref{PSBKV80} we obtain
\begin{align}\label{PSBKV42}
\left|\int_{x}^{\xi_{n}}\!\!\!\!\int_{t}^{\xi_{n}}\frac{\lambda_{n}^{-\gamma}}{1+i\lambda_{n}a(s)}f_{2,n}''(s)\,\ud s\, \ud t\right|&\leq\lambda_{n}^{-\gamma}\|f_{2,n}''\|_{L^{2}(0,1)}\int_{x}^{\xi_{n}}\!\!\!\left(\int_{t}^{\xi_{n}}\frac{\ud s}{1+\lambda_{n}^{2}a^{2}(s)}\right)^{\frac{1}{2}}\ud t\nonumber
\\
&=\left\{\begin{array}{ll}
o\left(\lambda_{n}^{-\frac{3}{2}\delta-\gamma}\right)&\text{if }1\leq \alpha\delta
\\
o\left(\lambda_{n}^{-\frac{2+\delta}{2(1+2\alpha)}-\delta-\gamma}\right)&\text{if }1>\alpha\delta.
\end{array}\right.
\end{align}
From \eqref{PSBKV73}, \eqref{PSBKV75} and \eqref{PSBKV78}, we have
\begin{align}\label{PSBKV43}
\left|M_{n}'(\xi_{n})\int_{x}^{\xi_{n}}\!\!\!\!\int_{t}^{\xi_{n}}\frac{i\lambda_{n}}{1+i\lambda_{n}a(s)}(\xi_{n}-s)\,\ud s\,\ud t\right|&=\lambda_{n}^{1-\delta}|M_{n}'(\xi_{n})|\int_{x}^{\xi_{n}}\!\!\!\!\int_{t}^{\xi_{n}}\frac{1}{(1+\lambda_{n}^{2}a^{2}(s))^{\frac{1}{2}}}\,\ud s\,\ud t\nonumber
\\
&=|M_{n}'(\xi_{n})|\times\left\{\begin{array}{ll}
O\left(\lambda_{n}^{1-3\delta}\right)&\text{if }1\leq \alpha\delta
\\
o\left(\lambda_{n}^{1-\frac{1+\delta}{1+\alpha}-2\delta}\right)&\text{if }1>\alpha\delta
\end{array}\right.\nonumber
\\
&=\left\{\begin{array}{ll}
o\left(\lambda_{n}^{1-3\delta+p_{0}}\right)&\text{if }1\leq \alpha\delta
\\
o\left(\lambda_{n}^{1-\frac{1+\delta}{1+\alpha}-2\delta+p_{0}}\right)&\text{if }1>\alpha\delta.
\end{array}\right.
\end{align}
Using the same arguments as in \eqref{PSBKV43} we can prove also that
\begin{equation}\label{PSBKV44}
\left|M_{n}(\xi_{n})\int_{x}^{\xi_{n}}\!\!\!\!\int_{t}^{\xi_{n}}\frac{i\lambda_{n}}{1+i\lambda_{n}a(s)}\,\ud s\,\ud t\right|=\left\{\begin{array}{ll}
o\left(\lambda_{n}^{1-2\delta+p_{0}}\right)&\text{if }1\leq \alpha\delta
\\
o\left(\lambda_{n}^{1-\frac{1+\delta}{1+\alpha}-\delta+p_{0}}\right)&\text{if }1>\alpha\delta.
\end{array}\right.
\end{equation}
By the Cauchy-Schwarz inequality, Lemma \ref{PSBKV27}, Lemma \ref{PSBKV106} and \eqref{PSBKV11} we obtain
\begin{align}\label{PSBKV81}
\left|\int_{0}^{\xi_{n}}v_{2,n}(s)\,\ud s\right|&\leq\left(\int_{0}^{\xi_{n}}s^{-\beta}\,\ud s\right)^{\frac{1}{2}}\left(\int_{0}^{1}x^{\beta}|v_{2,n}(s)|^{2}\,\ud s\right)^{\frac{1}{2}}\nonumber
\\
&\leq C\xi_{n}^{\frac{1-\beta}{2}}\|x^{\frac{\beta}{2}}v_{2,n}\|_{L^{2}(0,1)}\nonumber
\\
&\leq C\lambda_{n}^{-\frac{1-\beta}{2}\delta}\|x^{\frac{\overline{\beta}}{2}}v_{2,n}'\|_{L^{2}(0,1)}\nonumber
\\
&\leq C\lambda_{n}^{-\frac{1-\beta}{2}\delta}\|x^{\frac{\alpha}{2}}v_{2,n}''\|_{L^{2}(0,1)}\nonumber
\\
&=o\left(\lambda_{n}^{-\frac{1-\beta}{2}\delta-\frac{\gamma}{2}}\right)
\end{align}
where $-1<\beta<1$ and $\overline{\beta}>-1$ are satisfying
\begin{align}
\beta\geq \overline{\beta}-2 \text{ if }\overline{\beta}>1, \quad\beta=0\text{ if }\overline{\beta}=1\quad\text{ and }\quad\beta>-1 \text{ if } 0\leq \overline{\beta}<1\label{PSBKV108}
\\
\overline{\beta}\geq \alpha-2 \text{ if }1<\alpha<5, \quad \overline{\beta}=0\text{ if }\alpha=1 \quad\text{ and }\quad\overline{\beta}>-1 \text{ if } 0\leq \alpha<1.\label{PSBKV109}
\end{align}
Due to \eqref{PSBKV75}, \eqref{PSBKV77} and \eqref{PSBKV81} one gets
\begin{align}\label{PSBKV82}
\left|\int_{x}^{\xi_{n}}\frac{i\lambda_{n}}{\sqrt{1+i\lambda_{n}a(\tau)}}\int_{\tau}^{\xi_{n}}v_{2,n}(s)\,\ud s\,\ud\tau\right|&\leq o\left(\lambda_{n}^{1-\frac{1-\beta}{2}\delta-\frac{\gamma}{2}}\right)\int_{0}^{\xi_{n}}\frac{\ud \tau}{\left(1+\lambda_{n}^{2}a^{2}(\tau)\right)^{\frac{1}{4}}}\nonumber
\\
&=\left\{\begin{array}{ll}
o\left(\lambda_{n}^{1-\frac{3-\beta}{2}\delta-\frac{\gamma}{2}}\right)&\text{if } 1\leq\alpha\delta
\\
o\left(\lambda_{n}^{1-\frac{1-\beta}{2}\delta-\frac{1+2\delta}{2+\alpha}-\frac{\gamma}{2}}\right)&\text{if } 1>\alpha\delta.
\end{array}\right.
\end{align}
We use the same arguments as in \eqref{PSBKV81} then by \eqref{PSBKV75} and \eqref{PSBKV78} we obtain
\begin{align}\label{PSBKV83}
\left|\int_{x}^{\xi_{n}}\!\!\!\!\int_{\tau}^{\xi_{n}}\frac{i\lambda_{n}}{\sqrt{1+i\lambda_{n}a(\rho)}}v_{2,n}(\rho)\,\ud\rho\,\ud\tau\right|&\leq\lambda_{n}\int_{x}^{\xi_{n}}\!\!\!\!\int_{\tau}^{\xi_{n}}\frac{\rho^{-\frac{\beta'}{2}}}{(1+\lambda_{n}^{2}a^{2}(\rho))^{\frac{1}{4}}}\rho^{\frac{\beta}{2}}v_{2,n}(\rho)\,\ud\rho\,\ud\tau\nonumber
\\
&\leq\lambda_{n}^{1-\delta}\left(\int_{0}^{\xi_{n}}\frac{\rho^{-\beta}}{(1+\lambda_{n}^{2}a^{2}(\rho))^{\frac{1}{2}}}\,\ud\rho\right)^{\frac{1}{2}}\left(\int_{0}^{1}\rho^{\beta}|v_{2,n}(\rho)|^{2}\,\ud \rho\right)^{\frac{1}{2}}\nonumber
\\
&=\|s^{\frac{\alpha}{2}}v_{2,n}''\|_{L^{2}(0,1)}\left\{\begin{array}{ll}
O\left(\lambda_{n}^{1-\frac{3-\beta}{2}\delta}\right)&\text{if }1\leq \alpha\delta
\\
O\left(\lambda_{n}^{1-\frac{1-\alpha\delta}{2(1+\alpha)}-\frac{3-\beta}{2}\delta}\right)&\text{if }1>\alpha\delta
\end{array}\right.\nonumber
\\
&=\left\{\begin{array}{ll}
o\left(\lambda_{n}^{1-\frac{3-\beta}{2}\delta-\frac{\gamma}{2}}\right)&\text{if }1\leq \alpha\delta
\\
o\left(\lambda_{n}^{1-\frac{1-\alpha\delta}{2(1+\alpha)}-\frac{3-\beta}{2}\delta-\frac{\gamma}{2}}\right)&\text{if }1>\alpha\delta
\end{array}\right.
\end{align}
where $-1<\beta<1$ satisfying \eqref{PSBKV108} and \eqref{PSBKV109}.
\\
From \eqref{PSBKV75} and \eqref{PSBKV81} we have
\begin{equation}\label{PSBKV84}
\left|\frac{i\lambda}{\sqrt{1+i\lambda_{n}a(x)}}\int_{x}^{\xi_{n}}\!\!\!\!\int_{t}^{\xi_{n}}v_{2,n}(s)\,\ud s\,\ud t\right|=o\left(\lambda_{n}^{1-\frac{3-\beta}{2}\delta-\frac{\gamma}{2}}\right).
\end{equation}
Performing an integration by parts and uses \eqref{PSBKV82} and \eqref{PSBKV83} one finds
\begin{align}\label{PSBKV46}
&\left|\int_{x}^{\xi_{n}}\!\!\!\!\int_{\tau}^{\xi_{n}}\left(\frac{i\lambda_{n}}{\sqrt{1+i\lambda_{n}a(\rho)}}\right)'\int_{\rho}^{\xi_{n}}v_{2,n}(s)\,\ud s\,\ud\rho\,\ud\tau\right|\nonumber
\\
&\leq\left|\int_{x}^{\xi_{n}}\frac{i\lambda_{n}}{\sqrt{1+i\lambda_{n}a(\tau)}}\int_{\tau}^{\xi_{n}}v_{2,n}(s)\,\ud s\,\ud\tau\right|+\left|\int_{x}^{\xi_{n}}\!\!\!\!\int_{\tau}^{\xi_{n}}\frac{i\lambda_{n}}{\sqrt{1+i\lambda_{n}a(\rho)}}v_{2,n}(\rho)\,\ud\rho\,\ud\tau\right|\nonumber
\\
&=\left\{\begin{array}{ll}
o\left(\lambda_{n}^{1-\frac{3-\beta}{2}\delta-\frac{\gamma}{2}}\right)&\text{if }1\leq \alpha\delta
\\
o\left(\lambda_{n}^{1-\frac{1-\alpha\delta}{2(1+\alpha)}-\frac{3-\beta}{2}\delta-\frac{\gamma}{2}}\right)&\text{if }1>\alpha\delta
\end{array}\right.
\end{align}
and from \eqref{PSBKV83} we have
\begin{align}\label{PSBKV49}
\left|\int_{x}^{\xi_{n}}\!\!\!\!\int_{\tau}^{\xi_{n}}\left(\frac{i\lambda_{n}}{\sqrt{1+i\lambda_{n}a(\rho)}}\right)''\int_{\rho}^{\xi_{n}}\!\!\!\!\int_{t}^{\xi_{n}}v_{2,n}(s)\,\ud s\,\ud t\,\ud \rho\,\ud\,\tau\right|&\leq \left|\int_{x}^{\xi_{n}}\!\!\!\!\int_{\tau}^{\xi_{n}}\frac{i\lambda_{n}}{\sqrt{1+i\lambda_{n}a(\rho)}}v(\rho)\,\ud\rho\,\ud\tau\right|\nonumber
\\
+2\left|\int_{x}^{\xi_{n}}\frac{i\lambda_{n}}{\sqrt{1+i\lambda_{n}a(\tau)}}\int_{t}^{\xi_{n}}v(s)\,\ud s\,\ud \tau\right|&+\left|\frac{i\lambda_{n}}{\sqrt{1+i\lambda_{n}a(x)}}\int_{x}^{\xi_{n}}\!\!\!\!\int_{t}^{\xi_{n}}v(s)\,\ud s\,\ud t\right|\nonumber
\\
&=o\left(\lambda_{n}^{1-\frac{3-\beta}{2}\delta-\frac{\gamma}{2}}\right).
\end{align}
Now we set
$$
\ds p=\max\left\{p_{0},\,1-\frac{7}{2}\delta-\gamma,\,-\frac{3}{2}\delta-\gamma,\,1-\frac{3-\beta}{2}\delta-\frac{\gamma}{2},\,1-2\delta+p_{0}\right\}
$$
if $\alpha\delta\geq1$ and
$$
\ds p=\max\left\{p_{0},\,1-\frac{1+\delta}{1+\alpha}-\frac{5}{2}\delta-\gamma,\,-\frac{2+\delta}{2(1+2\alpha)}-\delta-\gamma,\,1-\frac{1+\delta}{1+\alpha}-2\delta+p_{0},\,1-\frac{3-\beta}{2}\delta-\frac{\gamma}{2}\right\}
$$
if $\alpha\delta<1$ where $-1<\beta<1$ satisfying \eqref{PSBKV108} and \eqref{PSBKV109}.
\\
Then by combining \eqref{PSBKV47} with \eqref{PSBKV73}, \eqref{PSBKV45}, \eqref{PSBKV41}, \eqref{PSBKV42}, \eqref{PSBKV43}, \eqref{PSBKV44}, \eqref{PSBKV46} and \eqref{PSBKV49} we find
\begin{equation}\label{PSBKV50}
\max_{x\in[0,\xi_{n}]}\left|z_{n,0}^{\pm}(x)\right|=o\left(\lambda_{n}^{p}\right).
\end{equation}
We fix $m\in\N$ and we suppose that for any $n\in\N$ we have
\begin{equation}\label{PSBKV52}
\max_{x\in[0,\xi_{n}]}\left|z_{n,m}^{\pm}(x)\right|=\left\{\begin{array}{ll}
o\left(\lambda_{n}^{p+(1-2\delta)m}\right)&\text{if }1\leq\alpha\delta
\\
o\left(\lambda_{n}^{p+\frac{\alpha-\delta(\alpha+2)}{1+\alpha}m}\right)&\text{if }1>\alpha\delta.
\end{array}\right.
\end{equation}
Then following to \eqref{PSBKV48}, \eqref{PSBKV75} and \eqref{PSBKV78} we obtain
\begin{align*}
\max_{x\in[0,\xi_{n}]}\left|z_{n,m+1}^{\pm}(x)\right|&=\max_{x\in[0,\xi_{n}]}\left|\int_{x}^{\xi_{n}}\!\!\!\int_{t}^{\xi_{n}}\frac{i\lambda_{n}}{1+i\lambda_{n}a(s)}z_{n,m}^{\pm}(s)\,\ud s\,\ud t\right|
\\
&=\left\{\begin{array}{ll}
o\left(\lambda_{n}^{p+(1-2\delta)(m+1)}\right)&\text{if }1\leq\alpha\delta
\\
o\left(\lambda_{n}^{p+\frac{\alpha-\delta(\alpha+2)}{1+\alpha}(m+1)}\right)&\text{if }1>\alpha\delta.
\end{array}\right.
\end{align*}
Thus by induction this proves that \eqref{PSBKV52} holds true for all $n,\,m\in\N$. Moreover, providing that \eqref{PSBKV86} and
\begin{equation}\label{PSBKV51}
\frac{\alpha}{\alpha+2}<\delta<\frac{1}{\alpha}\;\text{ if } 0<\alpha<2 \qquad \text{or }\qquad \delta>\frac{1}{2}\;\text{ and } \delta\geq\frac{1}{\alpha}
\end{equation}
hold true  the series defined by $\ds\sum_{m\geq 0}z_{n,m}^{\pm}$ is normally convergent in $[0,\xi_{n}]$.

We set
$$
\tilde{z}_{n}^{\pm}(x)=\sum_{m=0}^{+\infty}z_{n,m}^{\pm}(x).
$$
From \eqref{PSBKV52} and taking into account \eqref{PSBKV51} it follows that
\begin{align}\label{PSBKV53}
\max_{x\in[0,\xi_{n}]}\left|\tilde{z}_{n}^{\pm}(x)\right|&=o\left(\lambda_{n}^{p}\right).\left\{\begin{array}{ll}
\ds\sum_{n=0}^{+\infty}\lambda_{n}^{(1-2\delta)m}&\text{if }1\leq\alpha\delta
\\
\ds\sum_{n=0}^{+\infty}\lambda_{n}^{\frac{\alpha-\delta(\alpha+2)}{1+\alpha}m}&\text{if }1\leq\alpha\delta
\end{array}\right.\nonumber
\\
&=o\left(\lambda_{n}^{p}\right).\left\{\begin{array}{ll}
\ds\frac{1}{1-\lambda_{n}^{(1-2\delta)}}&\text{if }1\leq\alpha\delta
\\
\ds\frac{1}{1-\lambda_{n}^{\frac{\alpha-\delta(\alpha+2)}{1+\alpha}}}&\text{if }1\leq\alpha\delta
\end{array}\right.\nonumber
\\
&=o\left(\lambda_{n}^{p}\right).
\end{align}
Besides, by induction we can prove that $z_{n,m}^{\pm}$ is of class $\mathcal{C}^{2}$ in $[0,\xi_{n}]$ and we have 
\begin{equation*}
\left(z_{n,m+1}^{\pm}\right)'(x)=\mp\int_{x}^{\xi_{n}}\frac{i\lambda_{n}}{1+i\lambda_{n}a(s)}z_{n,m}^{\pm}(s)\,\ud s\quad\forall\,x\in[0,\xi_{n}]\;\forall\,n,\,m\in\N
\end{equation*}
and
\begin{equation*}
\left(z_{n,m+1}^{\pm}\right)''(x)=\pm\frac{i\lambda_{n}}{1+i\lambda_{n}a(x)}z_{n,m}^{\pm}(x)\quad\forall\,x\in[0,\xi_{n}]\;\forall\,n,\,m\in\N.
\end{equation*}
Therefore, following to \eqref{PSBKV75}, \eqref{PSBKV78} and \eqref{PSBKV52} for all $m\geq1$ we have
\begin{equation*}
\max_{x\in[\varepsilon,\xi_{n}]}\left|\left(z_{n,m}^{\pm}\right)'(x)\right|=\left\{\begin{array}{ll}
o\left(\lambda_{n}^{p+1-\delta-\delta+(1-2\delta)(m-1)}\right)&\text{if }1\leq\alpha\delta
\\
o\left(\lambda_{n}^{p+\frac{\alpha-\delta}{1+\alpha}+\frac{\alpha-(\alpha+2)\delta}{1+\alpha}(m-1)}\right)&\text{if }1>\alpha\delta
\end{array}\right.
\end{equation*}
and
\begin{equation*}
\max_{x\in[\varepsilon,\xi_{n}]}\left|\left(z_{n,m}^{\pm}\right)''(x)\right|=\left\{\begin{array}{ll}
o\left(\lambda_{n}^{p+1+(1-2\delta)(m-1)}\right)&\text{if }1\leq\alpha\delta
\\
o\left(\lambda_{n}^{p+1+\frac{\alpha-(\alpha+2)\delta}{1+\alpha}(m-1)}\right)&\text{if }1>\alpha\delta.
\end{array}\right.
\end{equation*}
which by taking into account \eqref{PSBKV51} implies that the series $\ds\sum_{m\geq 1}^{+\infty}\left(z_{n,m}\right)'(x)$ and $\ds\sum_{m\geq 1}^{+\infty}\left(z_{n,m}\right)''(x)$ are normally in $[0,\xi_{n}]$ and consequently $\tilde{z}_{m}^{\pm}$ is  of class $\mathcal{C}^{2}$ in $[0,\xi_{n}]$ and we have
\begin{align*}
\left(\tilde{z}^{\pm}_{n}\right)''(x)&=\sum_{m=0}^{+\infty}\left(z_{n,m}^{\pm}\right)''(x)=\pm\left(z_{n,0}^{\pm}\right)''(x)\pm\frac{i\lambda_{n}}{1+i\lambda_{n}a(x)}\sum_{m=0}^{+\infty}z_{n,m}^{\pm}(x)
\\
&=\left(\frac{i\lambda_{n}}{\sqrt{1+i\lambda_{n}a(x)}}\right)''\int_{x}^{\xi_{n}}\!\!\!\!\int_{t}^{\xi_{n}}v_{2,n}(s)\,\ud s\,\ud t-2\left(\frac{i\lambda_{n}}{\sqrt{1+i\lambda_{n}a(x)}}\right)'\int_{x}^{\xi_{n}}v_{2,n}(s)\,\ud s\nonumber
\\
&\pm\frac{i\lambda_{n}}{1+i\lambda_{n}a(x)}M_{n}(\xi_{n})\pm\frac{i\lambda_{n}}{1+i\lambda_{n}a(x)}(\xi_{n}-x)M_{n}'(\xi_{n})\nonumber
\\
&\pm\frac{\lambda_{n}^{1-\gamma}}{1+i\lambda_{n}a(x)}\int_{x}^{\xi_{n}}\!\!\!\!\int_{t}^{\xi_{n}}g_{2,n}(s)\,\ud s\,\ud t\mp\frac{\lambda_{n}^{-\gamma}}{1+i\lambda_{n}a(x)}f_{2,n}''(x)\pm\frac{i\lambda_{n}}{1+i\lambda_{n}a(x)}.\tilde{z}_{n}^{\pm}(x)
\end{align*}
Consequently, $\tilde{z}^{\pm}_{n}$ is solution of the ordinary differential equation \eqref{PSBKV25}. Hence by uniqueness of the solution we deduce that for every $n\in\N$ $z^{\pm}_{n}=\tilde{z}^{\pm}_{n}$ and by \eqref{PSBKV53}, we have
\begin{equation}\label{PSBKV54}
\max_{x\in[0,\xi_{n}]}\left|z_{n}^{\pm}(x)\right|=o(\lambda_{n}^{p}).
\end{equation}

We set
$$
m_{n}=\max_{x\in[0,\xi_{n}]}\left\{\left|z_{n}^{+}(x)\right|+\left|z_{n}^{-}(x)\right|\right\}.
$$
Since we have
$$
v_{2,n}(x)=\frac{1}{2}\left(z_{n}^{+}(x)-z_{n}^{-}(x)\right)\;\forall\, x\in[0,\xi_{n}]
$$
then from \eqref{PSBKV54} we find
\begin{equation}\label{PSBKV55}
\|v_{2,n}\|_{L^{2}(0,\xi_{n})}=\frac{1}{2}\|\left(z_{n}^{+}(x)-z_{n}^{-}(x)\right)\|_{L^{2}(0,\xi_{n})}\leq m_{n}\xi_{n}^{\frac{1}{2}}=o\left(\lambda_{n}^{p-\frac{\delta}{2}}\right)
\end{equation}
and
\begin{equation}\label{PSBKV56}
|v_{2,n}(0)|=\frac{1}{2}\left|z_{n}^{+}(0)-z_{n}^{-}(0)\right|\leq m_{n}=o\left(\lambda_{n}^{p}\right).
\end{equation}
Integrating \eqref{PSBKV7} over $(0,\xi_{n})$ we obtain
\begin{equation}\label{PSBKV57}
i\lambda_{n}\int_{0}^{\xi_{n}}v_{2,n}(s)\,\ud s+M_{n}'(\xi_{n})-M_{n}'(0)=\lambda_{n}^{-\gamma}\int_{0}^{\xi_{n}}g_{2,n}(s)\,\ud s
\end{equation}
From \eqref{PSBKV81} and \eqref{PSBKV7} we have respectively
\begin{equation}\label{PSBKV58}
\lambda_{n}\left|\int_{0}^{\xi_{n}}v_{2,n}(s)\,\ud s\right|=o\left(\lambda_{n}^{1-\frac{1-\beta}{2}\delta-\frac{\gamma}{2}}\right)
\end{equation}
and
\begin{equation}\label{PSBKV59}
\lambda_{n}^{-\gamma}\left|\int_{0}^{\xi_{n}}g_{2,n}(s)\,\ud s\right|\leq \lambda_{n}^{-\gamma}\sqrt{\xi_{n}}\|g_{2,n}\|_{L^{2}(0,\xi_{n})}=o\left(\lambda_{n}^{-(\gamma+\frac{\delta}{2})}\right)
\end{equation}
then by combining \eqref{PSBKV57} with \eqref{PSBKV73}, \eqref{PSBKV58} and \eqref{PSBKV59} we arrive
\begin{equation}\label{PSBKV60}
|M_{n}'(0)|=\left(\lambda_{n}^{p_{0}}+\lambda_{n}^{1-\frac{1-\beta}{2}\delta-\frac{\gamma}{2}}+\lambda_{n}^{-(\gamma+\frac{\delta}{2})}\right)o(1).
\end{equation}
Integrating \eqref{PSBKV7} tow times, one gets
\begin{equation}\label{PSBKV61}
i\lambda_{n}\int_{0}^{\xi_{n}}\!\!\!\!\int_{0}^{t}v_{2,n}(s)\,\ud s\,\ud t+M_{n}(\xi_{n})-M_{n}(0)-M_{n}'(0)\xi_{n}=\lambda_{n}^{-\gamma}\int_{0}^{\xi_{n}}\!\!\!\!\int_{0}^{t}g_{2,n}(s)\,\ud s\,\ud t.
\end{equation}
Then by combining \eqref{PSBKV61} with \eqref{PSBKV73}, \eqref{PSBKV58}, \eqref{PSBKV59} and \eqref{PSBKV60} we have
\begin{equation}\label{PSBKV62}
|M_{n}(0)|=\left(\lambda_{n}^{p_{0}}+\lambda_{n}^{1-\frac{3-\beta}{2}\delta-\frac{\gamma}{2}}+\lambda_{n}^{-(\gamma+\frac{3\delta}{2})}\right)o(1).
\end{equation}
Uses the trace formula and \eqref{PSBKV4} one has
\begin{equation}\label{PSBKV70}
|u_{2,n}(0)|+|u_{2,n}(\xi_{n})|+|u_{2,n}'(0)|+|u_{2,n}'(\xi_{n})|\leq  C\|u_{2,n}\|_{H^{2}(0,1)}=O(1).
\end{equation}
Using the transmission conditions \eqref{PSBKV19}, \eqref{PSBKV21} and \eqref{PSBKV22} to substitute \eqref{PSBKV56}, \eqref{PSBKV60}, \eqref{PSBKV62} and \eqref{PSBKV70} into \eqref{PSBKV13} we follow
\begin{equation}\label{PSBKV63}
\|u_{1,n}''\|_{L^{2}(-1,0)}+\|v_{1,n}\|_{L^{2}(-1,0)}=o(1)
\end{equation}
providing that \eqref{PSBKV86}, \eqref{PSBKV51} hold true and
\begin{equation}\label{PSBKV65}
p\leq 0\qquad\text{and}\qquad 1-\frac{1-\beta}{2}\delta-\frac{\gamma}{2}\leq0
\end{equation}
where $-1<\beta<1$ satisfying \eqref{PSBKV108} and \eqref{PSBKV109}.

Since $a$ is an increasing function, we get
\begin{align*}
\|\sqrt{a}u_{2,n}''\|_{L^{2}(0,1)}&\geq\|\sqrt{a}u''\|_{L^{2}(\xi_{n},1)}
\\
&\geq\min_{\xi_{n}\leq x\leq 1}\left(\sqrt{a(x)}\right)\|u_{2,n}''\|_{L^{2}(\xi_{n},1)}
\\
&\geq\sqrt{a(\xi_{n})}\|u_{2,n}''\|_{L^{2}(\xi_{n},1)}
\\
&\geq C \xi_{n}^{\frac{\alpha}{2}}\|u_{2,n}''\|_{L^{2}(\xi_{n},1)}
\\
&\geq C \lambda_{n}^{-\frac{\alpha\delta}{2}}\|u_{2,n}''\|_{L^{2}(\xi_{n},1)}.
\end{align*}
Then from \eqref{PSBKV12} we obtain
\begin{equation}\label{PSBKV64}
\|u_{2,n}''\|_{L^{2}(\xi_{n},1)}=o\left(\lambda_{n}^{\frac{\alpha\delta}{2}-\frac{\gamma}{2}-1}\right).
\end{equation}
Multiplying \eqref{PSBKV7} by $\lambda_{n}^{-\gamma}\overline{u}_{2,n}$ and integration over $(0,\xi_{n})$, by an integration by parts we arrive 
\begin{align}\label{PSBKV66}
\|u_{2,n}''\|_{L^{2}(0,\xi_{n})}^{2}&=-i\lambda_{n}\int_{0}^{\xi_{n}}v_{2,n}(x).\overline{u}_{2,n}(x)\,\ud x-\int_{0}^{\xi_{n}} a(x)v_{2,n}''(x).\overline{u}_{2,n}(x)\,\ud x
\\
&-M_{n}'(\xi_{n})\overline{u}_{2,n}(\xi_{n})+M_{n}'(0)\overline{u}_{2,n}(0)+M_{n}(\xi_{n})\overline{u}_{2,n}'(\xi_{n})-M_{n}(0)\overline{u}_{2,n}'(0)+o\left(\lambda_{n}^{-\gamma}\right).\nonumber
\end{align}
From Lemma \ref{PSBKV27}, Lemma \ref{PSBKV106}, \eqref{PSBKV11} and \eqref{PSBKV12}, we have 
\begin{align}\label{PSBKV74}
\lambda_{n}\left|\int_{0}^{\xi_{n}}v_{2,n}(x).\overline{u}_{2,n}(x)\,\ud x\right|&=\lambda_{n}\left|\int_{0}^{\xi_{n}}x^{\frac{\beta_{0}'}{2}}v_{2,n}(x).x^{-\frac{\beta_{0}'}{2}}\overline{u}_{2,n}(x)\,\ud x\right|\nonumber
\\
&\leq\lambda_{n}\|x^{\frac{\beta_{0}'}{2}}v_{2,n}\|_{L^{2}(0,1)}\|x^{-\frac{\beta_{0}'}{2}}u_{2,n}\|_{L^{2}(0,1)}\nonumber
\\
&\leq C\lambda_{n}\|x^{\frac{\beta_{0}''}{2}}v_{2,n}'\|_{L^{2}(0,1)}\|x^{\frac{\beta_{0}''}{2}}u_{2,n}'\|_{L^{2}(0,1)}\nonumber
\\
&\leq C\lambda_{n}\|x^{\frac{\alpha}{2}}v_{2,n}''\|_{L^{2}(0,1)}\|x^{\frac{\alpha}{2}}u_{2,n}''\|_{L^{2}(0,1)}\nonumber
\\
&=o\left(\lambda_{n}^{-\gamma}\right).
\end{align}
where $0\leq\beta_{0}'<1$ and $\beta_{0}''>-1$  can be chosen in such a way
\begin{align}
\beta_{0}'\geq \beta_{0}''-2 \text{ if }\beta_{0}''>1, \quad\beta_{0}'=0\text{ if }\beta_{0}''=1\quad\text{ and }\quad\beta_{0}'>-1 \text{ if } 0\leq \beta_{0}''<1\label{PSBKV110}
\\
\beta_{0}''\geq \alpha-2 \text{ if }1<\alpha<5, \quad \beta_{0}''=0\text{ if }\alpha=1 \quad\text{ and }\quad\beta_{0}''>-1 \text{ if } 0\leq \alpha<1.\label{PSBKV111}
\end{align}
Using \eqref{PSBKV4}, \eqref{PSBKV11} and \eqref{PSBKV18} we obtain
\begin{align}\label{PSBKV71}
\left|\int_{0}^{\xi_{n}} a(x)v_{2,n}''(x).\overline{u}_{2,n}''(x)\,\ud x\right|&\leq \max_{0\leq x\leq \xi_{n}}\{\sqrt{a(x)}\}\|u_{2,n}''\|_{L^{2}(0,\xi_{n})}.\|a^{\frac{1}{2}}v_{2,n}''\|_{L^{2}(0,\xi_{n})}\nonumber
\\
&\leq \lambda_{n}^{-\alpha\delta}\|u_{2,n}''\|_{L^{2}(0,\xi_{n})}.\|a^{\frac{1}{2}}v_{2,n}''\|_{L^{2}(0,\xi_{n})}\nonumber
\\
&=o\left(\lambda_{n}^{-\alpha\delta-\frac{\gamma}{2}-1}\right).
\end{align}
Inserting \eqref{PSBKV73}, \eqref{PSBKV60}, \eqref{PSBKV62}, \eqref{PSBKV70}, \eqref{PSBKV74} and \eqref{PSBKV71} into \eqref{PSBKV66} we arrive at
\begin{equation}\label{PSBKV67}
\|u_{2,n}''\|_{L^{2}(0,\xi_{n})}=o(1)
\end{equation}
providing that \eqref{PSBKV33}, \eqref{PSBKV30}, \eqref{PSBKV86}, \eqref{PSBKV108}, \eqref{PSBKV109}, \eqref{PSBKV51} and \eqref{PSBKV65} hold true.
\\
So that, \eqref{PSBKV64} and \eqref{PSBKV67} leads to
\begin{equation}\label{PSBKV68}
\|u_{2,n}''\|_{L^{2}(0,1)}=o(1),
\end{equation}
providing that \eqref{PSBKV33}, \eqref{PSBKV30}, \eqref{PSBKV86}, \eqref{PSBKV108}, \eqref{PSBKV109}, \eqref{PSBKV51}, \eqref{PSBKV65} and
\begin{equation}\label{PSBKV85}
\frac{\alpha\delta}{2}-\frac{\gamma}{2}-1\leq 0
\end{equation}
hold true.
\\
The combination of \eqref{PSBKV63} and \eqref{PSBKV68} with \eqref{PSBKV17} leads to
\begin{equation}\label{PSBKV69}
\|v_{2,n}\|_{L^{2}(0,1)}=o(1).
\end{equation}
Thus, due to \eqref{PSBKV63}, \eqref{PSBKV68} and \eqref{PSBKV69} we follow
$$
\|(u_{n},v_{n})\|=o(1)
$$
which is a contradiction with \eqref{PSBKV4} as long as  $\gamma$, $\delta$, $\beta_{0}$, $\beta'$, $\beta$ satisfying conditions \eqref{PSBKV33}, \eqref{PSBKV30}, \eqref{PSBKV86}, \eqref{PSBKV108}, \eqref{PSBKV109}, \eqref{PSBKV51}, \eqref{PSBKV65} and \eqref{PSBKV85} hold true.
\section{Determination of the decay rate}\label{DTBKV}
To conclude the proof of Theorem \ref{IBKV3} it remains to prove that $\ds\gamma=\frac{2}{\tau(\alpha)}+\varepsilon$ for every $\varepsilon>0$. To do this we want to make $\gamma$ minimal subject to the constraints \eqref{PSBKV86}, \eqref{PSBKV51}, \eqref{PSBKV65} and \eqref{PSBKV85}. In other words, we need to solve the following ``optimality problem''
\begin{equation}\label{PSBKV87}
\left\{\begin{array}{l}
\inf(\gamma)
\\
-1<\beta<1,\;\beta'>-1,\;\beta_{0}>-1
\\
p_{0}\leq 0
\\
\ds \left(\frac{\alpha}{\alpha+2}<\delta<\frac{1}{\alpha}\text{ and } 0<\alpha<2\right) \text{ or } \left(\delta>\frac{1}{2}\text{ and } \delta\geq\frac{1}{\alpha}\right)
\\
p\leq 0
\\
\ds\frac{\alpha\delta}{2}-\frac{\gamma}{2}-1\leq 0
\\
\ds1-\frac{1-\beta}{2}\delta-\frac{\gamma}{2}\leq 0
\end{array}\right.
\end{equation}
where $\beta_{0}$ is satisfying \eqref{PSBKV31}, $\beta'$ is satisfying \eqref{PSBKV33}-\eqref{PSBKV30} and $\beta$ is satisfying \eqref{PSBKV108}-\eqref{PSBKV109}.

Several cases are discussed in the sequel. We note that by the definition of $p$ and due to the therd and the seventh lines of \eqref{PSBKV87} condition $p\leq 0$ can be canceled out. Hence, taking into account the expression of $p_{0}$, \eqref{PSBKV87} is written as follow
\begin{equation}\label{PSBKV93}
\left\{\begin{array}{l}
\inf(\gamma)
\\
\ds \left(\frac{\alpha}{\alpha+2}<\delta<\frac{1}{\alpha}\text{ and } 0<\alpha<2\right) \text{ or } \left(\delta>\frac{1}{2}\text{ and } \delta\geq\frac{1}{\alpha}\right)
\\
-1<\beta<1,\;\beta'>-1,\;\beta_{0}>-1
\\
\gamma\geq (1+\beta')\delta
\\
\gamma\geq (1+\beta_{0})\delta
\\
\gamma\geq (\alpha+1)\delta-2
\\
\gamma\geq (1-\alpha)\delta
\\
\ds\gamma\geq (\beta'-1)\delta+2
\\
\ds\gamma\geq (\alpha+3)\delta-2
\\
\ds\gamma\geq (3-\alpha)\delta
\\
\gamma\geq \alpha\delta-2
\\
\gamma\geq (\beta-1)\delta+2.
\end{array}\right.
\end{equation}
It is important for the sequel to notice that $\beta_{0}$ can be chosen as near as we want to $-1$ if $0<\alpha<1$, equal to $0$ if $\alpha=1$ and equal to $\alpha-2$ if $1<\alpha<5$. Concerning $\beta'$ or $\beta$ there are chosen as near as we want to $-1$ if $0<\alpha\leq3$ and equal to $\alpha-4$ if $3<\alpha<5$, in particular we can suppose $\beta=\beta'$. Hence \eqref{PSBKV93} is reduced to
\begin{equation}\label{PSBKV88}
\left\{\begin{array}{l}
\inf(\gamma)
\\
\ds \left(\frac{\alpha}{\alpha+2}<\delta<\frac{1}{\alpha}\text{ and } 0<\alpha<2\right) \text{ or } \left(\delta>\frac{1}{2}\text{ and } \delta\geq\frac{1}{\alpha}\right)
\\
-1<\beta<1,\;\beta_{0}>-1
\\
\gamma\geq (1+\beta)\delta
\\
\gamma\geq (1+\beta_{0})\delta
\\
\gamma\geq (\beta-1)\delta+2
\\
\gamma\geq (\alpha+3)\delta-2
\\
\gamma\geq (3-\alpha)\delta.
\end{array}\right.
\end{equation}
 
\underline{\textbf{Case 1:} $0<\alpha\leq1$.} Let's first focus on the case $\ds\frac{\alpha}{\alpha+2}<\delta<\frac{1}{\alpha}$. Since $\beta$ and $\beta'$ are arbitrary small real numbers strictly greater than $-1$ then \eqref{PSBKV88} is reduced to the following problem
\begin{equation}\label{PSBKV89}
\left\{\begin{array}{l}
\inf(\gamma)
\\
\gamma>0 \text{ if } 0<\alpha<1\text{ and }\gamma\geq\delta \text{ if }\alpha=1
\\
\ds\frac{\alpha}{\alpha+2}<\delta<\frac{1}{\alpha}
\\
\ds\gamma>-2\delta+2
\\
\ds\gamma\geq (\alpha+3)\delta-2
\\
\ds\gamma\geq (3-\alpha)\delta.
\end{array}\right.
\end{equation}
Thanks to the second and the ninth lines of \eqref{PSBKV89} the problem is simply written
\begin{equation}\label{PSBKV90}
\left\{\begin{array}{l}
\inf(\gamma)
\\
\ds\frac{\alpha}{\alpha+2}<\delta< \frac{1}{\alpha}
\\
\ds \gamma\geq (3-\alpha)\delta=\gamma_{1}(\delta)
\\
\gamma> -2\delta+2=\gamma_{2}(\delta).
\end{array}\right.
\end{equation}
Noting that $\gamma_{1}(\delta)\geq\gamma_{2}(\delta)$ if and only if $\ds \delta\geq \frac{2}{5-\alpha}$ and $\ds \frac{\alpha}{\alpha+2}<\frac{2}{5-\alpha}<1$. Then since  $\gamma_{1}$ is increasing and $\gamma_{2}$ is decreasing we have 
\begin{equation}\label{PSBKV91}
\ds\inf_{\frac{\alpha}{\alpha+2}<\delta<\frac{1}{\alpha}}(\gamma)=\gamma_{1}\left(\frac{2}{5-\alpha}\right)=\gamma_{2}\left(\frac{2}{5-\alpha}\right)=\frac{2(3-\alpha)}{5-\alpha}.
\end{equation}
In the other hand, taking into account that $\gamma_{1}$ is an increasing function then 
$$
\inf_{\delta\geq\frac{1}{\alpha}}(\gamma)\geq \gamma_{1}\left(\frac{1}{\alpha}\right)>\gamma_{1}\left(\frac{2}{5-\alpha}\right).
$$
which implies obviously that $\ds \inf(\gamma)=\frac{2(3-\alpha)}{5-\alpha}$ and therefore due to \eqref{PSBKV90} and \eqref{PSBKV91} we only need to take
$$
\ds\gamma=\frac{2(3-\alpha)}{5-\alpha}+\varepsilon
$$
where $\varepsilon$ is an arbitrary small strictly non-negative real number.

\underline{\textbf{Case 2:} $1<\alpha<2$.} Here we take $\beta_{0}=\alpha-2$ and $\beta>-1$ is arbitrary small. Let's first focus on the case where $\ds\frac{\alpha}{\alpha+2}<\delta<\frac{1}{\alpha}$, then problem \eqref{PSBKV88} is written as follow 
\begin{equation}\label{PSBKV92}
\left\{\begin{array}{l}
\inf(\gamma)
\\
\ds\frac{\alpha}{\alpha+2}<\delta<\frac{1}{\alpha}
\\
\gamma\geq (\alpha-1)\delta
\\
\ds\gamma>-2\delta+2
\\
\ds\gamma\geq (\alpha+3)\delta-2
\\
\ds\gamma\geq (3-\alpha)\delta.
\end{array}\right.
\end{equation}
With the help of the second and the sixth lines of \eqref{PSBKV92} this problem is recast
\begin{equation}\label{PSBKV94}
\left\{\begin{array}{l}
\inf(\gamma)
\\
\ds\frac{\alpha}{\alpha+2}<\delta<\frac{1}{\alpha}
\\
\ds\gamma\geq (3-\alpha)\delta=\gamma_{1}(\delta)
\\
\gamma> -2\delta+2=\gamma_{2}(\delta).
\end{array}\right.
\end{equation}
As it is noted above $\gamma_{1}(\delta)\geq\gamma_{2}(\delta)$ if and only if $\ds \delta\geq\frac{2}{5-\alpha}$. However, $\ds \frac{\alpha}{\alpha+2}<\frac{2}{5-\alpha}<\frac{1}{\alpha}$ if $\ds 1<\alpha<\frac{5}{3}$, otherwise $\ds\frac{1}{\alpha}\leq\frac{2}{5-\alpha}$. Hence by using the same arguments as previously we 
\begin{equation}\label{PSBKV100}
\ds\inf(\gamma)=\gamma_{1}\left(\frac{2}{5-\alpha}\right)=\gamma_{2}\left(\frac{2}{5-\alpha}\right)=\frac{2(3-\alpha)}{5-\alpha}\quad\text{if }1<\alpha<\frac{5}{3}.
\end{equation}
and 
\begin{equation}\label{PSBKV114}
\ds\inf_{\frac{\alpha}{\alpha+2}<\delta<\frac{1}{\alpha}}(\gamma)=\gamma_{2}\left(\frac{1}{\alpha}\right)=\frac{2(\alpha-1)}{\alpha}\quad\text{if }\frac{5}{3}\leq\alpha<2.
\end{equation}
We consider now the sub-case where $\ds\delta\geq \frac{1}{\alpha}$ and $\ds\frac{5}{3}\leq\alpha<2$, then \eqref{PSBKV88} is written as follow
\begin{equation}\label{PSBKV95}
\left\{\begin{array}{l}
\inf(\gamma)
\\
\ds\delta\geq\frac{1}{\alpha}
\\
\gamma\geq (\alpha-1)\delta
\\
\ds\gamma>-2\delta+2
\\
\ds\gamma\geq (\alpha+3)\delta-2
\\
\ds\gamma\geq (3-\alpha)\delta.
\end{array}\right.
\end{equation}
Due to the second and the fifth lines of \eqref{PSBKV95}, we reduce the problem to the following
\begin{equation}\label{PSBKV97}
\left\{\begin{array}{l}
\inf(\gamma)
\\
\ds\delta\geq\frac{1}{\alpha}
\\
\gamma\geq (\alpha+3)\delta-2=\gamma_{3}(\delta)
\\
\ds\gamma>-2\delta+2=\gamma_{2}(\delta).
\end{array}\right.
\end{equation}
Noting that $\gamma_{2}(\delta)\geq\gamma_{1}(\delta)$ if and only if $\ds\delta\geq \frac{4}{5+\alpha}$ and $\ds\frac{1}{\alpha}\leq \frac{4}{5+\alpha}$ if $\ds\alpha\geq\frac{5}{3}$. Since $\gamma_{3}$ is increasing and $\gamma_{2}$ is decreasing we deduce that
\begin{equation}\label{PSBKV98}
\begin{array}{ll}
\ds\inf_{\delta\geq\frac{1}{\alpha}}(\gamma)=\gamma_{2}\left(\frac{4}{5+\alpha}\right)=\gamma_{3}\left(\frac{4}{5+\alpha}\right)=\frac{2(1+\alpha)}{5+\alpha}&\ds\text{if }\frac{5}{3}\leq\alpha< 2.
\end{array}
\end{equation}
Therefore by combining \eqref{PSBKV114} and \eqref{PSBKV98} we find
\begin{equation}\label{PSBKV99}
\begin{array}{ll}
\ds\inf(\gamma)=\min\left\{\gamma_{2}\left(\frac{1}{\alpha}\right),\gamma_{2}\left(\frac{4}{5+\alpha}\right)\right\}=\gamma_{2}\left(\frac{4}{5+\alpha}\right)=\frac{2(1+\alpha)}{5+\alpha}&\ds\text{if }\frac{5}{3}\leq\alpha<2.
\end{array}
\end{equation}
So that, from \eqref{PSBKV100} and \eqref{PSBKV99} we only have to take
\begin{equation}\label{PSBKV96}
\gamma=\left\{\begin{array}{ll}
\ds\frac{2(3-\alpha)}{5-\alpha}+\varepsilon&\ds\text{if }1<\alpha<\frac{5}{3}
\\
\ds\frac{2(1+\alpha)}{5+\alpha}+\varepsilon&\ds\text{if }\frac{5}{3}\leq\alpha<2. 
\end{array}\right.
\end{equation}
where $\varepsilon$ is an arbitrary small strictly positive real number.

\underline{\textbf{Case 3:} $2\leq \alpha\leq3$.} Here we take $\beta_{0}=\alpha-2$ and $\beta>-1$ is small. In this case since $\alpha>2$ then \eqref{PSBKV88} is written as follow
\begin{equation}\label{PSBKV102}
\left\{\begin{array}{l}
\inf(\gamma)
\\
\ds\delta>\frac{1}{2}
\\
\gamma\geq (\alpha-1)\delta
\\
\gamma> -2\delta+2
\\
\gamma\geq (\alpha+3)\delta-2
\\
\gamma\geq (3-\alpha)\delta.
\end{array}\right.
\end{equation}
From the second and the fifth lines of \eqref{PSBKV102} we can reduce the problem is equivalent
\begin{equation}\label{PSBKV116}
\left\{\begin{array}{l}
\inf(\gamma)
\\
\ds\delta>\frac{1}{2}
\\
\gamma> -2\delta+2=\gamma_{2}(\delta)
\\
\gamma\geq (\alpha+3)\delta-2=\gamma_{3}(\delta).
\end{array}\right.
\end{equation}
Since in this case we have $\ds\frac{4}{5+\alpha}>\frac{1}{2}$ then by using the same arguments as previously we arrive 
\begin{equation}\label{PSBKV104}
\begin{array}{ll}
\ds\inf(\gamma)=\gamma_{2}\left(\frac{4}{5+\alpha}\right)=\gamma_{4}\left(\frac{4}{5+\alpha}\right)=\frac{2(1+\alpha)}{5+\alpha}&\ds\text{if }2\leq \alpha<3.
\end{array}
\end{equation}
Hence, \eqref{PSBKV104} and \eqref{PSBKV116} give
\begin{equation*}
\gamma=\begin{array}{ll}
\ds\frac{2(1+\alpha)}{5+\alpha}+\varepsilon&\ds\text{if }2\leq\alpha\leq3
\end{array}
\end{equation*}
where $\varepsilon$ is an arbitrary small strictly positive real number.

\underline{\textbf{Case 4:} $3<\alpha<5$.} By taking $\beta_{0}=\alpha-2$ and $\beta=\alpha-4$ then \eqref{PSBKV87} is equivalent to
\begin{equation}\label{PSBKV103}
\left\{\begin{array}{l}
\inf(\gamma)
\\
\ds\delta>\frac{1}{2}
\\
\gamma\geq (\alpha-1)\delta
\\
\gamma\geq (\alpha-3)\delta
\\
\gamma\geq (\alpha-5)\delta+2
\\
\gamma\geq (\alpha+3)\delta-2
\\
\gamma\geq (3-\alpha)\delta.
\end{array}\right.
\end{equation}
Due to the second and the seventh lines of \eqref{PSBKV103} the problem is reduced to the following
\begin{equation}\label{PSBKV115}
\left\{\begin{array}{l}
\inf(\gamma)
\\
\ds\delta>\frac{1}{2}
\\
\gamma\geq (\alpha+3)\delta-2=\gamma_{3}.
\end{array}\right.
\end{equation}
Since $\gamma_{3}$ is an increasing function, one gets
\begin{equation}\label{PSBKV105}
\begin{array}{ll}
\ds\inf(\gamma)=\gamma_{3}\left(\frac{1}{2}\right)=\frac{\alpha-1}{2}&\text{if }3<\alpha<5.
\end{array}
\end{equation}
Hence, from \eqref{PSBKV115} and \eqref{PSBKV105} we deduce
\begin{equation*}
\gamma=\begin{array}{ll}
\ds\frac{\alpha-1}{2}+\varepsilon&\ds\text{if }3<\alpha<5. 
\end{array}
\end{equation*}
where $\varepsilon$ is an arbitrary small strictly positive real number. This completes the proof.

\end{document}